\documentclass[11pt]{article}
\usepackage{amssymb}
\makeatletter\@addtoreset {equation}{section}\makeatother

\newtheorem{theorem}{Theorem}
\newtheorem{lemma}{Lemma}
\newtheorem{remark}{Remark}

\newtheorem{proposition}{Proposition}

\newenvironment{proof}{
    \noindent {\it Proof.}}{\hfill$\Box$
}

\usepackage[dvips]{epsfig}
\usepackage{graphicx}

\begin{document}

\title{\bf Global solutions to the shallow-water system}

\author{Sergey N. Alexeenko$^1$, Marina V. Dontsova$^2$, and Dmitry E. Pelinovsky$^{1,3}$ \\
{\small $^1$ Department of Applied Mathematics, Nizhny Novgorod State Technical University} \\
{\small 24 Minin Street, Nizhny Novgorod, 603950, Russia} \\
{\small $^2$ Department of Applied Mathematics, Nizhny Novgorod State Pedagogicall University} \\
{\small 1 Ulyanova street, Nizhny Novgorod, 603950, Russia} \\
{\small $^3$ Department of Mathematics, McMaster University, Hamilton, Ontario, Canada, L8S 4K1}}

\date{\today}
\maketitle

\begin{abstract}
The classical system of shallow-water (Saint--Venant) equations describes long
surface waves in an inviscid incompressible fluid of a variable depth.
Although shock waves are expected in this quasilinear hyperbolic system for
a wide class of initial data, we find a sufficient condition on the initial data that guarantees existence of
a global classical solution continued from a local solution. The sufficient conditions
can be easily satisfied for the fluid flow propagating in one direction with two characteristic velocities
of the same sign and two monotonically increasing Riemann invariants. We prove that these properties
persist in the time evolution of the classical solutions to the shallow-water equations and provide no shock wave singularities formed
in a finite time over a half-line or an infinite line. On a technical side, we develop a novel method of an additional argument,
which allows to obtain local and global solutions to the quasilinear hyperbolic systems in physical rather than characteristic variables.
\end{abstract}

\section{Introduction}

The shallow water system arises in the dispersionless limit of Euler equations and describes
long waves on the surface of an inviscid incompressible fluid (e.g., water). We assume that the surface waves
are two-dimensional in the $(x,z)$-variables and that the fluid is located between the hard bottom of the varying depth
at $z = -h(x)$ and the free surface at $z = \eta(t,x)$, where $h$ is given and $\eta$ is unknown.

In the case of surface waves free of vorticity, the velocity vector of the fluid's motion
is given by the gradient of the velocity potential, which is found from the Laplace equation in
variables $(x,z)$. In the dispersionless limit, for which the horizontal length of wave motion is much larger compared to the
vertical length, the Euler equations reduce to the shallow water system
\begin{equation}
\label{shallow-water}
\left\{ \begin{array}{l}
\partial_t \eta + \partial_x \left[ u (h(x) + \eta) \right] = 0, \\
\partial_t u + u \partial_x u + g \partial_x \eta = 0, \end{array} \right.
\end{equation}
where $u(t,x)$ is the horizontal component of velocity at the free surface $z = \eta(t,x)$,
and $g$ is the gravitational constant. In what follows, we set $g = 1$ without loss of generality.

The shallow water system (\ref{shallow-water}), which is also known as the Saint--Venant equations,
is reviewed in many texts and monographs (see, e.g., Section 5.1.1 in \cite{Lannes}).
Recently, interest to the shallow-water system arises due to modeling of run-up of water waves towards the beach \cite{Diden2}.
In particular, when the bottom topography changes like $h(x) \sim x^{4/3}$, the waves propagating towards the beach
are free of reflections \cite{Diden1}.

Using the standard technique of Riemann invariants, one can diagonalize the quasilinear system (\ref{shallow-water})
in new coordinates
\begin{equation}
\label{RI}
z_{\pm}(t,x) := u(t,x) \pm 2\sqrt{h(x) + \eta(t,x)},
\end{equation}
which are real if $h(x) + \eta(t,x) > 0$. This constraint corresponds
to the hyperbolicity of the shallow-water system (\ref{shallow-water}) and, physically, to the nonzero depth
of the fluid flow over the variable bottom. Substitution of (\ref{RI}) into (\ref{shallow-water})
yields the system of symmetric quasilinear equations
\begin{equation}
\label{sym-system}
\left\{ \begin{array}{l}
\partial_t z_+ + \frac{1}{4} (3z_+ + z_-) \partial_x z_+ = h'(x), \\
\partial_t z_- + \frac{1}{4} (z_+ + 3 z_- ) \partial_x z_- = h'(x). \end{array} \right.
\end{equation}
The characteristic speeds of the system (\ref{sym-system}) are given by
\begin{equation}
\label{speeds}
c_{\pm} := \frac{1}{4} (3 z_{\pm} + z_{\mp}) = u \pm \sqrt{h(x) + \eta}.
\end{equation}
System (\ref{sym-system}) in Riemann invariants is well-known, see, e.g., Sections 5.7 and 13.10 in \cite{Whitham}.
Some explicit solutions can be obtained in the case $h'(x) = {\rm const}$
by using the hodograph transformation method, see, e.g., recent works \cite{Diden2,Rybkin} and references therein.
Review of exact solutions to the shallow water system can be found in Section 16.2.1 in \cite{Polyanin}.

The Cauchy problem is posed for the system (\ref{sym-system}) with initial data
\begin{equation}
\label{Cauchy-value}
z_{\pm}(0,x) = \varphi_{\pm}(x).
\end{equation}
If the initial data $\varphi_{\pm}$ are defined on the infinite
line in Sobolev spaces $H^s(\mathbb{R})$, then the Cauchy problem (\ref{sym-system}) and (\ref{Cauchy-value})
is known to be locally well-posed for $s > \frac{3}{2}$ \cite{Kato}.
The method of characteristics can be used in a local neighborhood of any point if the initial data $\varphi_{\pm}$
are $C^1$ functions near this point and their first derivatives are Lipschitz continuous \cite{CourantLax}.

It is easy to recover the solution $(u,\eta)$ to the shallow-water system (\ref{shallow-water})
from the solution $(z_+,z_-)$ to the system (\ref{sym-system}). Indeed, inverting (\ref{RI}) yields
\begin{equation}
\label{RH-inverse}
u(t,x) = \frac{1}{2} \left[ z_+(t,x) + z_-(t,x) \right], \quad
\eta(t,x) = \frac{1}{16} \left[ z_+(t,x) - z_-(t,x) \right]^2 - h(x).
\end{equation}
The initial data for $u$ and $\eta$ are given by
\begin{equation}
\label{RH-inverse-initial}
u_0(x) = \frac{1}{2} \left[ \varphi_+(x) + \varphi_-(x) \right], \quad
\eta_0(x) = \frac{1}{16} \left[ \varphi_+(x) - \varphi_-(x) \right]^2 - h(x),
\end{equation}
where positivity of $h(x) + \eta_0(x) > 0$ is assumed for every $x$.

For most quasilinear systems, local solutions in Sobolev spaces $H^s(\mathbb{R})$
are not continued for all times $t$ because wave breaking occurs in a finite time, resulting
in appearance of the shock waves \cite{Dafermos}. However, depending on the initial values $\varphi_{\pm}$
and the given profile $h$, the wave breaking may be avoided and the local solutions
can be continued for all finite times. We term such solutions as global solutions and warn that
these solutions are allowed to diverge in some norm as $t \to \infty$.

This paper is devoted to the solvability of the classical system (\ref{sym-system})
both locally and globally. We will consider the semi-infinite line $[0,\infty)$ for $x$. Generally speaking,
a boundary condition is required at the finite boundary $x = 0$ for all positive times $t > 0$.
However, if we find a condition on the initial values $\varphi_{\pm}$ and the given profile $h$
which ensure that both characteristic speeds $c_{\pm}$ in (\ref{speeds}) are negative near $x = 0$
for all $t > 0$, then we can avoid setting boundary conditions at $x = 0$. This is the key ingredient of the method
of an additional argument, which we develop in this work. Moreover, with additional
constraints on $\varphi_{\pm}$ and $h$, one can also continue classical solutions to
the shallow-water system (\ref{shallow-water}) globally in time and thus avoid wave breaking.

The novel method of an additional argument was pioneered for scalar
conservation laws in \cite{Iman1,Iman2} and for systems of conservation laws in \cite{Alex1,Alex2}.
This method allows us to avoid technical problems arising in other techniques
such as the method of characteristics or the method of generalized solutions \cite{RY}.
For instance, the solvability condition in the method of characteristics relies on
invertibility of the characteristic variables,
which may be difficult to prove. Compared to the method of characteristics, the method
of an additional argument allows us to obtain the local and global solvability of the quasilinear system
directly in physical coordinates.

In what follows, for a given $T > 0$, we use notation
$$
\Omega_T := \{ (t,x) : \;\; t \in (0,T), \;\; x \in \mathbb{R}^+ \},
$$
for the domain of definition of the Cauchy problem associated with the system (\ref{sym-system}).
We denote by $C^{1,1}(\Omega_T)$ the space of bounded functions of two variables in $\Omega_T$,
which are continuously differentiable both in $t$ and $x$ with bounded first derivatives.
We also introduce the norm in the space of functions $C_b^n(\mathbb{R}^+)$
with bounded and continuous derivatives up to the $n$-th order:
$$
\| h \|_{C_b^n} := \sup_{x \in \mathbb{R}^+} |h(x)| + \sum_{j=1}^n \sup_{x \in \mathbb{R}^+} |h^{(j)}(x)|,
\quad h \in C_b^n(\mathbb{R}^+).
$$
The following two theorems present the main results obtained in this paper.

\begin{theorem}
Assume that $u_0, \eta_0 \in C^1_b(\mathbb{R}^+)$ and $h \in C^2_b(\mathbb{R}^+)$ satisfy the conditions
\begin{equation}
\label{cond1}
h(x) \geq 0, \quad h'(x) \leqslant 0, \quad x \in \mathbb{R}^+,
\end{equation}
and
\begin{equation}
\label{cond2}
\eta_0(x) \geq C, \quad u_0(x) \leqslant  - 2\sqrt {h(x) + {\eta _0}(x)},  \quad x \in \mathbb{R}^+,
\end{equation}
for a fixed positive constant $C$. 
Then, for every $T > 0$ satisfying the constraint
\begin{equation}
\label{time-constraint-local}
T \leqslant \min\left(\frac{C_\varphi}{C_h},\frac{1}{15 C_\varphi}\right),
\end{equation}
where $C_h := \| h \|_{C^2_b}$ and $C_{\varphi} := \max\{\| \varphi_+ \|_{C^1_b},\| \varphi_- \|_{C^1_b}\}$
with the initial data $\varphi_{\pm} := u_0 \pm \sqrt{h + \eta}$, there exists a unique classical solution
$u,\eta \in C^{1,1}(\Omega _T)$ to the shallow-water system (\ref{shallow-water})
such that $u|_{t = 0} = u_0$ and $\eta |_{t =0} = \eta_0$. 
\label{theorem-main-local}
\end{theorem}

\begin{theorem}
In addition to the conditions of Theorem \ref{theorem-main-local}, assume that
$u_0, \eta_0 \in C^1_b(\mathbb{R}^+)$ and $h \in C^2_b(\mathbb{R}^+)$ satisfy the conditions
\begin{equation}
\label{cond3-on-h}
h''(x) \geq 0, \quad x \in \mathbb{R}^+
\end{equation}
and
\begin{equation}
\label{cond3}
u_0'(x) \geq \frac{|h'(x) + \eta_0'(x)|}{\sqrt{h(x) + {\eta _0}(x)}},  \quad x \in \mathbb{R}^+.
\end{equation}
Then, for every $T > 0$, there exists a unique classical solution $u,\eta \in C^{1,1}(\Omega _T)$
to the shallow-water system (\ref{shallow-water}) such that $u|_{t = 0} = u_0$ and $\eta |_{t =0} = \eta_0$.
\label{theorem-main-global}
\end{theorem}

\begin{remark}
It follows from the definition  (\ref{RI}) for Riemann invariants that
conditions (\ref{cond2}) are satisfied if
\begin{equation}
\label{cond-phi-intro-1}
\varphi_+(x) \leq 0, \quad \varphi_-(x) \leq 0, \quad x \in \mathbb{R}^+.
\end{equation}
Similarly, condition (\ref{cond3}) is satisfied if
\begin{equation}
\label{cond-phi-intro-2}
\varphi_+'(x) \geq 0, \quad \varphi_-'(x) \geq 0, \quad x \in \mathbb{R}^+.
\end{equation}
\end{remark}

\begin{remark}
Since the quasilinear system (\ref{sym-system}) is written in the symmetric form,
the result of Theorem \ref{theorem-main-local} agrees with the result of Kato \cite{Kato} on the infinite line, since Sobolev space
$H^s(\mathbb{R})$ with $s > \frac{3}{2}$ is continuously embedded into the space $C^1_b(\mathbb{R})$.
However, the Cauchy problem (\ref{sym-system}) and (\ref{Cauchy-value}) on the half-line cannot be solved
by the method of Kato \cite{Kato} unless a boundary condition is set at $x = 0$ in one way or another.
\end{remark}

\begin{remark}
The result of Theorem \ref{theorem-main-local} is stronger than the corresponding result of Courant and Lax \cite{CourantLax},
which establish the existence of classical solutions with Lipschitz continuity for their spatial derivatives
in a local neighborhood of any point $x$ on $\mathbb{R}^+$, provided the initial data are available near
this point. Although the formulations of the method of characteristics in \cite{CourantLax} and
the method of an additional argument here are similar, our technique allows us to obtain the solution
to the quasilinear system (\ref{sym-system}) in physical rather than characteristic coordinates.
Also we obtain a stronger result by using the Schauder fixed point theorem (see Lemma \ref{lemma-implicit} below)
instead of the Arzel\'a--Ascoli theorem on convergence of bounded and equicontinuous sequences for spatial derivatives.
\end{remark}

The alternative of the global existence in Theorem \ref{theorem-main-global} is the wave breaking in a finite time, which
happens when the shock waves are formed in the quasilinear hyperbolic systems \cite{Dafermos}.
We note that the wave breaking can also occur in the presence of
weak dispersion, if the initial data are sufficiently large in some norm \cite{Pelin1,Pelin2}.

The physical relevance of the conditions (\ref{cond1}), (\ref{cond2}), (\ref{cond3-on-h}), and (\ref{cond3}) is to
provide the bottom topography $h$ and the initial values for $u$ and $\eta$
such that both the Riemann invariants $z_{\pm}$ and their
characteristic speeds $c_{\pm}$ given by (\ref{RI}) and (\ref{speeds})
are strictly negative, whereas the Riemann invariants are monotonically increasing,
see (\ref{cond-phi-intro-1}) and (\ref{cond-phi-intro-2}). Under these conditions,
the surface waves do not break in a finite time, because they move convectively
to the finite boundary at $x = 0$, through which they radiate away.
These conditions can be satisfied, for instance, if
\begin{equation}
\label{example-main}
h(x) = (1+x)^{-p}, \quad \eta_0(x) = C, \quad u_0(x) = -2 \sqrt{C + h(x)},
\end{equation}
where $p > 0$ and $C > 0$ are fixed.
Thus, the bottom topography becomes deeper near $x = 0$ and uniform as $x \to \infty$,
whereas the initial horizontal velocity is negative everywhere and the current is stronger
near $x = 0$ and becomes uniform as $x \to \infty$. Such configurations can model
river waterfalls, e.g., Niagara falls in Ontario, Canada.

Theorems \ref{theorem-main-local} and \ref{theorem-main-global} can be extended to
the infinite line without any restrictions, as long as the conditions (\ref{cond1}),
(\ref{cond2}), (\ref{cond3-on-h}),  and (\ref{cond3}) hold on the infinite line. The main example (\ref{example-main})
does not make sense on the infinite line, but the conditions can be satisfied for
the shear flow on the flat bottom with sign-definite, monotonically increasing velocity $u_0$, which may vanish at
one infinity but has a non-vanishing background flow at the other infinity.

In a single wave reduction of the system (\ref{sym-system}) with $h'(x) \equiv 0$ and
$z_+(t,x) \equiv 0$, the constraint (\ref{cond3}) guarantees that
$$
\varphi_-'(x) = u_0'(x) - \frac{\eta_0'(x)}{\sqrt{h + \eta_0(x)}} \geq 0, \quad x \in \mathbb{R}^+.
$$
This condition is well known \cite{Dafermos} to exclude shock waves in the Cauchy problem
posed for the inviscid Burgers equation
\begin{equation}
\label{Burgers}
\left\{ \begin{array}{l}
\partial_t z_- + \frac{3}{4} z_- \partial_x z_- = 0, \\
z_- |_{t = 0} = \varphi_-. \end{array} \right.
\end{equation}
In the same context, the constraint (\ref{cond2}) ensures that $\varphi_-(x) \leq 0$ for every $x \in \mathbb{R}^+$,
the latter constraint is only required to avoid the boundary condition on $z_-$
at $x = 0$ for the evolution problem (\ref{Burgers}) on the semi-infinite line $\mathbb{R}^+$.

The rest of this paper is organized as follows. Section 2 is devoted
to the reformulation of the quasilinear system (\ref{sym-system}) as a system of
integral equations by using the method of an additional argument. The equivalence between
the quasilinear system (\ref{sym-system}) and the system of integral
equations is established. In Section 3, we obtain a local
solution of Theorem \ref{theorem-main-local}.
In Section 4, we show that
the local solution in $C^{1,1}(\Omega_T)$ can be extended for every $T > 0$
as in Theorem \ref{theorem-main-global}.
The additional constraints (\ref{cond3-on-h}) and (\ref{cond3}) allow us
to control the rate of change of the spatial derivatives of the solution $z_{\pm}$ during
the time evolution of the quasilinear system (\ref{sym-system}).

\section{Reformulation with the method of an additional argument}

Here we adopt the method of an additional argument in order to reformulate the
Cauchy problem given by  (\ref{sym-system}) and (\ref{Cauchy-value})
as a boundary-value problem along characteristic coordinates.
For a given point $(t,x) \in \Omega_T$, we introduce the extended characteristic coordinates $\eta_+(s;t,x)$
and $\eta_-(s;t,x)$  from solutions to the system of differential equations
\begin{eqnarray}
\label{coordinate-1-2}
\left\{ \begin{array}{l}
\frac{d \eta_+}{ds}(s;t,x) = \frac{1}{4} \left[ 3 z_+(s,\eta_+(s;t,x)) + z_-(s,\eta_+(s;t,x)) \right], \\
\frac{d \eta_-}{ds}(s;t,x) = \frac{1}{4} \left[ z_+(s,\eta_-(s;t,x)) + 3 z_-(s,\eta_-(s;t,x)) \right],
\end{array} \right. \quad 0 \leq s \leq t,
\end{eqnarray}
starting with the boundary values $\eta_{\pm}(t;t,x) = x$. In the characteristic variables,
the system (\ref{sym-system}) can be rewritten as the system of differential equations
\begin{eqnarray}
\label{evolution-1-2}
\left\{ \begin{array}{l}
\frac{d z_+}{ds}(s,\eta_+(s;t,x)) = h'(\eta_+(s;t,x)), \\
\frac{d z_-}{ds}(s,\eta_-(s;t,x)) = h'(\eta_-(s;t,x)), \end{array} \right. \quad 0 \leq s \leq t,
\end{eqnarray}
starting with the initial values $z_{\pm}(0,\eta_{\pm}(0;t,x)) = \varphi_{\pm}(\eta_{\pm}(0;t,x))$.
The domain of definition of the systems (\ref{coordinate-1-2}) and (\ref{evolution-1-2}) is given by
\begin{equation}
\Gamma_T := \left\{ (s,t,x) : \quad 0 \le s \le t \le T, \quad x \in \mathbb{R}^+ \right\},
\end{equation}
for a given $T > 0$. We denote by $C^{k,k,m}(\Gamma_T)$ the space of bounded functions of three variables in $\Gamma_T$,
which are differentiable $k$-times with respect to $s$ and $t$, $m$-times with respect to $x$,
with bounded derivatives. We also denote the supremum norm of a function $U \in C^{0,0,0}(\Gamma_T)$ by
\begin{equation}
\label{norm}
\| U \| := \sup_{(s,t,x) \in \Gamma_T} |U(s;t,x)|.
\end{equation}

The variable $s$ is referred to as the additional argument of the system (\ref{coordinate-1-2}) and (\ref{evolution-1-2}).
The main difference of the method of an additional argument from the method of characteristics is that the system
(\ref{coordinate-1-2}) is integrated backward in $s$ from the current time
$t$ to the initial time $0$, whereas the system (\ref{evolution-1-2}) is integrated forward in $s$
from the initial time $0$ to the current time $t$. Although the combined system (\ref{coordinate-1-2})
and (\ref{evolution-1-2}) represents a boundary-value problem instead of the Cauchy problem,
we are still able to rewrite the systems (\ref{coordinate-1-2}) and (\ref{evolution-1-2})
as a system of integral equations and to solve it by the Picard method of successful iterations.
Compared to the method of characteristics, the solutions $z_{\pm}(t,x) \equiv z_{\pm}(t,\eta_{\pm}(t;t,x))$
appear in physical rather than characteristic coordinates.

\subsection{Integral equations for classical solutions of system (\ref{sym-system})}

Integrating (\ref{coordinate-1-2}) backward in $s$, we obtain the system of integral equations
\begin{eqnarray}
\label{int-coordinate-1-2}
\left\{ \begin{array}{l}
\eta_+(s;t,x) = x - \frac{1}{4} \int_s^t \left[ 3 z_+(\nu,\eta_+(\nu;t,x)) + z_-(\nu,\eta_+(\nu;t,x)) \right] d\nu, \\
\eta_-(s;t,x) = x - \frac{1}{4} \int_s^t \left[ z_+(\nu,\eta_-(\nu;t,x)) + 3 z_-(\nu,\eta_-(\nu;t,x)) \right] d\nu,
 \end{array} \right. \quad 0 \leq s \leq t.
\end{eqnarray}
Integrating (\ref{evolution-1-2}) forward in $s$, we obtain another system of integral equations
\begin{eqnarray}
\label{int-evolution-1-2}
\left\{ \begin{array}{l}
z_+(s,\eta_+(s;t,x)) = \varphi_+(\eta_+(0;t,x)) + \int_0^s h'(\eta_+(\nu;t,x))) d\nu, \\
z_-(s,\eta_-(s;t,x)) = \varphi_-(\eta_-(0;t,x)) + \int_0^s h'(\eta_-(\nu;t,x))) d\nu,  \end{array} \right. \quad 0 \leq s \leq t,
\end{eqnarray}
From the geometric definition of the characteristic curves in the domain $\Omega_T$ on the $(t,x)$ plane, we have the connection
formulas
\begin{equation}
\label{connection-formula}
\left\{ \begin{array}{l}
z_-(s,\eta_+(s;t,x)) = z_-(s,\eta_-(s;s,\eta_+(s;t,x))), \\
z_+(s,\eta_-(s;t,x)) = z_+(s,\eta_+(s;s,\eta_-(s;t,x))), \end{array} \right. 
\quad 0 \leq s \leq t.
\end{equation}

Let us denote
\begin{equation}
\label{relations-Z-Y}
Z_{\pm}(s;t,x) := z_{\pm}(s,\eta_{\pm}(s;t,x)) \quad \mbox{\rm and} \quad
Y_{\pm}(s;t,x) := z_{\mp}(s,\eta_{\pm}(s;t,x)).
\end{equation}
It follows from the boundary conditions $\eta_{\pm}(t;t,x) = x$ that 
$Z_{\pm}(t;t,x) = z_{\pm}(t,x)$ and $Y_{\pm}(t;t,x) = z_{\mp}(t,x)$. 
Furthermore, equations (\ref{connection-formula}) are equivalent to
the following relations between variables $Z_{\pm}$ and $Y_{\pm}$:
\begin{eqnarray}
\label{correspondence}
Y_+(s;t,x) = Z_-(s;s,\eta_+(s;t,x)), \quad Y_-(s;t,x) = Z_+(s;s,\eta_-(s;t,x)).
\end{eqnarray}
From now on, we will be writing systems by using one equation with two subscripts.
By using new notations, we rewrite system (\ref{int-coordinate-1-2}) in the following form
\begin{eqnarray}
\label{int-coordinate-1-2-again}
\eta_{\pm}(s;t,x) = x - \frac{1}{4} \int_s^t \left[ 3 Z_{\pm}(\nu;t,x) + Y_{\pm}(\nu;t,x) \right] d\nu,  \quad 0 \leq s \leq t.
\end{eqnarray}
Therefore, the characteristic coordinates can be eliminated from the systems (\ref{int-evolution-1-2}) and
(\ref{correspondence}), after which we obtain the following integral equations
for unknown functions $Z_{\pm}$ and $Y_{\pm}$ in $\Gamma_T$:
\begin{eqnarray}
\nonumber
Z_{\pm}(s;t,x) & = & \varphi_{\pm}\left( x - \frac{1}{4} \int_0^t \left[ 3 Z_{\pm}(\nu;t,x) + Y_{\pm}(\nu;t,x) \right] d \nu \right) \\
\label{2.11} & \phantom{t} & +
\int_0^s h'\left(x - \frac{1}{4} \int_{\nu}^t \left[ 3 Z_{\pm}(\tau;t,x) + Y_{\pm}(\tau;t,x) \right] d\tau \right) d\nu,
\end{eqnarray}
and
\begin{eqnarray}
Y_{\pm}(s;t,x) = Z_{\mp}\left(s;s,x - \frac{1}{4} \int_s^t \left[ 3 Z_{\pm}(\nu;t,x) + Y_{\pm}(\nu;t,x) \right] d \nu \right).
\label{2.14}
\end{eqnarray}

Our first result states that the system of integral equations (\ref{2.11})--(\ref{2.14}) is closed in $\Gamma_T$ for every $T > 0$
under conditions (\ref{cond1}) and (\ref{cond-phi-intro-1}) on $h$ and $\varphi_{\pm}$.

\begin{proposition}
Assume that $h \in C^1_b(\mathbb{R}^+)$ and $\varphi_{\pm} \in C^0_b(\mathbb{R}^+)$.
Under the conditions
\begin{equation}
\label{cond-varphi}
h'(x) \leqslant 0, \quad \varphi_+(x) \leq 0, \quad \varphi_-(x) \leq 0, \quad x \in \mathbb{R}^+,
\end{equation}
the system of integral equations (\ref{2.11})--(\ref{2.14}) is closed in $\Gamma_T$ for every $T > 0$ in the sense
that if a unique solution $(Z_{\pm},Y_{\pm})$ exists in $C^{0,0,0}(\Gamma_T)$, then
\begin{equation}
\label{constraints-on-solutions}
\left\{ \begin{array}{l}
\eta_{\pm}(s;t,x) \geq 0, \\
Z_{\pm}(s;t,x) \leq 0, \\
Y_{\pm}(s;t,x) \leq 0, \end{array} \right.
\quad (s,t,x) \in \Gamma_T.
\end{equation}
\label{proposition-definition}
\end{proposition}

\begin{proof}
We obtain from (\ref{int-coordinate-1-2-again}), (\ref{2.11}), and (\ref{2.14}) for every $(s,t,x) \in \Gamma_T$,
$$
\eta_{\pm}(s;t,x) \geq x, \quad Z_{\pm}(s;t,x) \leq \varphi_{\pm}(\eta_{\pm}(0;t,x)) \leq 0, \quad
Y_{\pm}(s;t,x) = Z_{\mp}(s;s,\eta_{\pm}(s;t,x)) \leq 0,
$$
by using conditions (\ref{cond-varphi}) and the continuation arguments. Then, constraints (\ref{constraints-on-solutions})
follow.
\end{proof}

Next, we show how the classical solutions to the Cauchy problem (\ref{sym-system}) and (\ref{Cauchy-value})
are obtained from suitable solutions to the integral system  (\ref{2.11})--(\ref{2.14}).

\begin{proposition}
\label{prop-correspondence}
Assume that $h \in C^2_b(\mathbb{R}^+)$ and $\varphi_{\pm} \in C^1_b(\mathbb{R}^+)$.
If there exists a unique solution $(Z_{\pm},Y_{\pm}) \in C^{1,1,1}(\Gamma_T)$
of the system of integral equations (\ref{2.11})--(\ref{2.14}), then
$z_{\pm}(t,x) = Z_{\pm}(t;t,x)$ is a classical solution to system (\ref{sym-system})
in $C^{1,1}(\Omega_T)$ such that $z_{\pm}(0,x) = \varphi_{\pm}(x)$ for $x \in \mathbb{R}^+$.
\end{proposition}

\begin{proof}
Let us introduce two differential operators $W_{\pm}$ given by
$$
W_{\pm} f := \frac{\partial f}{\partial t} + \frac{1}{4} \left( 3 Z_{\pm}(t;t,x) + Z_{\mp}(t;t,x) \right) \frac{\partial f}{\partial x}.
$$
Applying $W_+$ to the corresponding integral equation in the system (\ref{2.11})
and using $Y_+(t;t,x) = Z_-(t;t,x)$ from (\ref{correspondence}), we obtain
\begin{eqnarray*}
(W_+ Z_+)(s;t,x) & = & -\frac{1}{4} \varphi_+'(\cdot)
\; \int_0^t \left[ 3 (W_+ Z_+)(s;t,x) + (W_+ Y_+)(s;t,x) \right] ds \\
& \phantom{t} & -\frac{1}{4}
\int_0^s h''(\cdot)
\left( \int_{\nu}^t \left[ 3 (W_+ Z_+)(\tau;t,x) + (W_+ Y_+)(\tau;t,x) \right] d\tau \right) d\nu,
\end{eqnarray*}
where the arguments of $\varphi_+'(\cdot)$ and $h''(\cdot)$ are the same as in (\ref{2.11}).
Since we have the correspondence between $Y_+$ and $Z_-$ from the system (\ref{2.14}), we obtain
similarly
\begin{eqnarray*}
(W_+ Y_+)(s;t,x) = -\frac{1}{4} \partial_x Z_-(\cdot) \;
\int_{s}^t \left[ 3 (W_+ Z_+)(\nu;t,x) + (W_+ Y_+)(\nu;t,x) \right] d\nu,
\end{eqnarray*}
where the argument of $Z_-(\cdot)$ is the same as in (\ref{2.14}).
By using the norm in $\Gamma_T$ defined by (\ref{norm}), we obtain the following estimate
$$
3 \| W_+ Z_+ \| + \| W_+ Y_+ \| \leq \frac{1}{4} \left( 3 \| \varphi_+ \|_{C^1_b} t +
\frac{3}{2} \| h \|_{C^2_b} t^2 + \| \partial_x Z_- \| t \right) \left( 3 \| W_+ Z_+ \| + \| W_+ Y_+ \| \right).
$$
Note that $\| \partial_x Z_- \| < \infty$ due to the assumption $Z_- \in C^{1,1,1}(\Gamma_T)$,
whereas $\| \varphi_+ \|_{C^1_b} < \infty$ and $\| h \|_{C^2_b} < \infty$ due to the assumptions on $\varphi_+$ and $h$.
Let $T_+$ be the smallest positive root of the algebraic equation
$$
\frac{1}{4} \left( 3 \| \varphi_+ \|_{C^1_b} t + \frac{3}{2} \| h \|_{C^2_b} t^2 + \| \partial_x Z_- \| t \right) = 1.
$$
Then, for every $t \in [0,t_+]$ with $t_+ := \min(T_+,T)$, we obtain
$$
\| W_+ Z_+ \| + \| W_+ Y_+ \| = 0,
$$
which imply $W_+ Z_+ = W_+ Y_+ = 0$ in $\Gamma_{t_+}$.

Applying $W_-$ to the corresponding integral equations into the system (\ref{2.11}) and (\ref{2.14}),
we obtain similar estimates
$$
3 \| W_- Z_- \| + \| W_- Y_- \| \leq \frac{1}{4}
\left( 3 \| \varphi_- \|_{C^1_b} t + \frac{3}{2} \| h \|_{C^2_b} t^2 + \| \partial_x Z_+ \| t \right)
\left( 3 \| W_- Z_- \| + \| W_- Y_- \| \right).
$$
Let $T_-$ be the smallest positive root of the algebraic equation
$$
\frac{1}{4} \left( 3 \| \varphi_- \|_{C^1_b} t + \frac{3}{2} \| h \|_{C^2_b} t^2 + \| \partial_x Z_+ \| t \right) = 1.
$$
Then, for every $t \in [0,t_-]$ with $t_- := \min(T_-,T)$, we obtain
$$
\| W_- Z_- \| + \| W_- Y_- \| = 0,
$$
which imply $W_- Z_- = W_- Y_- = 0$ in $\Gamma_{t_-}$.

Let $z_{\pm}(t,x) := Z_{\pm}(t;t,x)$ and $T_0 := \min(T_+,T_-,T)$. Then, for every $t \in [0,T_0]$,
we use $\partial_s Z_{\pm}(t;t,x) = h'(x)$ that follows from system (\ref{2.11}) and obtain
\begin{eqnarray*}
\frac{\partial z_+}{\partial t} + \frac{1}{4} (3 z_+ + z_-) \frac{\partial z_+}{\partial x} & = &
\frac{\partial Z_+}{\partial s}(t;t,x) + (W_+ Z_+)(t;t,x) = h'(x)
\end{eqnarray*}
and
\begin{eqnarray*}
\frac{\partial z_-}{\partial t} + \frac{1}{4} (z_+ + 3 z_-) \frac{\partial z_-}{\partial x} & = &
\frac{\partial Z_-}{\partial s}(t;t,x) + (W_- Z_-)(t;t,x) = h'(x),
\end{eqnarray*}
which is nothing but system (\ref{sym-system}). Therefore, $z_{\pm} \in C^{1,1}(\Omega_{T_0})$
is a solution to system (\ref{sym-system}) for $T_0 \leq T$. If $T_0 < T$, then the continuation
of the solution to the entire domain $\Omega_T$ can be performed in a finite number of steps.
\end{proof}

\subsection{Integral equations for $x$-derivatives of system (\ref{sym-system})}

Let us denote $u_{\pm}(t,x) := \partial_x z_{\pm}(t,x)$. If $z_{\pm} \in C^{1,1}(\Omega_T)$ as in
Proposition \ref{prop-correspondence}, then $u_{\pm} \in C^{0,0}(\Omega_T)$.
Differentiating (\ref{int-coordinate-1-2}) with respect to $x$, we obtain a system of integral equations
for $x$-derivatives of the characteristic coordinates:
\begin{eqnarray}
\label{coordinate-1-2-der}
\xi_{\pm}(s;t,x) = 1 - \frac{1}{4} \int_s^t \left[ 3 u_{\pm}(\nu,\eta_{\pm}(\nu;t,x)) + u_{\mp}(\nu,\eta_{\pm}(\nu;t,x)) \right] \xi_{\pm}(\nu;t,x) d \nu, 
\quad 0 \leq s \leq t, 
\end{eqnarray}
where $\xi_{\pm}(s;t,x) := \partial_x \eta_{\pm}(s;t,x)$ satisfies the initial conditions
$\xi_{\pm}(t;t,x) = 1$. There exists a unique solution of the system of integral equations
(\ref{coordinate-1-2-der}) in the form
\begin{equation}
\label{coordinate-1-2-der-sol}
\xi_{\pm}(s;t,x) = e^{-\frac{1}{4} \int_s^t \left[ 3 u_{\pm}(\nu,\eta_{\pm}(\nu;t,x)) + u_{\mp}(\nu,\eta_{\pm}(\nu;t,x)) \right] d\nu}, 
\quad 0 \leq s \leq t.
\end{equation}

The main difficulty in the method of characteristics is to control positivity of $\xi_{\pm}(s;t,x)$ in $\Gamma_T$ as $T$ increases.
The explicit expression (\ref{coordinate-1-2-der-sol}) shows that positivity of $\xi_{\pm}(s;t,x)$ in $\Gamma_T$
follows from boundness of $u_{\pm}(t,x)$ in $\Omega_T$, but this property is hard to control.
On the other hand, in the method of an additional argument, we introduce
\begin{equation}
\label{introduction-U-V}
U_{\pm}(s;t,x) := \partial_x Z_{\pm}(s;t,x), \quad
V_{\pm}(s;t,x) := \partial_x Y_{\pm}(s;t,x)
\end{equation}
and define by using (\ref{relations-Z-Y}) and the chain rule
\begin{equation}
\label{correspondence-U-V}
U_{\pm}(s;t,x) = u_{\pm}(s,\eta_{\pm}(s;t,x)) \xi_{\pm}(s;t,x), \quad
V_{\pm}(s;t,x) = u_{\mp}(s,\eta_{\pm}(s;t,x)) \xi_{\pm}(s;t,x).
\end{equation}
It follows from the boundary conditions $\eta_{\pm}(t;t,x) = x$ and $\xi_{\pm}(t;t,x) = 1$ 
that $U_{\pm}(t;t,x) = u_{\pm}(t,x)$ and $V_{\pm}(t;t,x) = u_{\mp}(t,x)$. 
If $(Z_{\pm},Y_{\pm}) \in C^{1,1,1}(\Gamma_T)$ as in Proposition \ref{prop-correspondence},
then $(U_{\pm},V_{\pm}) \in C^{0,0,0}(\Gamma_T)$.
By differentiating the system of integral equations (\ref{2.11}) and (\ref{2.14}) with respect to $x$,
we obtain the system of integral equations:
\begin{eqnarray}
\nonumber
U_{\pm}(s;t,x) & = & \varphi_{\pm}'\left(\cdot\right)
\left(  1 - \frac{1}{4} \int_0^t \left[ 3 U_{\pm}(\nu;t,x) + V_{\pm}(\nu;t,x) \right] d \nu \right)\\
\label{der-eq-1} & \phantom{t} & +
\int_0^s h''\left(\cdot\right) \left( 1 - \frac{1}{4} \int_{\nu}^t \left[ 3 U_{\pm}(\tau;t,x) + V_{\pm}(\tau;t,x) \right] d\tau  \right) d\nu,
\end{eqnarray}
and
\begin{eqnarray}
\nonumber
V_{\pm}(s;t,x) & = & U_{\mp}\left(s;s,x - \frac{1}{4} \int_s^t \left[ 3 Z_{\pm}(\nu;t,x) + Y_{\pm}(\nu;t,x) \right] d \nu \right) \\
& \phantom{t} & \phantom{t} \times
\left( 1 - \frac{1}{4} \int_s^t \left[ 3 U_{\pm}(\nu;t,x) + V_{\pm}(\nu;t,x) \right] d \nu  \right),
\label{der-eq-2}
\end{eqnarray}
where the arguments of $\varphi_{\pm}'(\cdot)$ and $h''(\cdot)$ are the same as in the integral equation (\ref{2.11}).
On the other hand, differentiating (\ref{int-coordinate-1-2-again}) in $x$ yields the following relation
\begin{equation}
\label{coordinate-1-2-der-sol-rel}
\xi_{\pm}(s;t,x) = 1 - \frac{1}{4} \int_s^t \left[ 3 U_{\pm}(\nu;t,x) + V_{\pm}(\nu;t,x) \right] d \nu, 
\quad 0 \leq s \leq t.
\end{equation}
This relation is complementary to the expression (\ref{coordinate-1-2-der-sol}).

The following proposition states that the variables $U_{\pm}$ and $V_{\pm}$ are sign-definite in $\Gamma_T$ for every $T > 0$,
for which a solution $(Z_{\pm},Y_{\pm}) \in C^{1,1,1}(\Gamma)$ exists,
under additional conditions (\ref{cond3-on-h}) and (\ref{cond-phi-intro-2}) on $h$ and $\varphi_{\pm}$.

\begin{proposition}
Assume that $h \in C^2_b(\mathbb{R}^+)$ and $\varphi_{\pm} \in C^1_b(\mathbb{R}^+)$ satisfy
(\ref{cond-varphi}), and the additional conditions
\begin{equation}
\label{cond3-final}
h''(x) \geq 0, \quad \varphi_+'(x) \geq 0, \quad \varphi'_-(x) \geq 0,  \quad x \in \mathbb{R}^+.
\end{equation}
If a solution $(Z_{\pm},Y_{\pm})$ to the system of integral equations (\ref{2.11})--(\ref{2.14}) exists in $C^{1,1,1}(\Gamma_T)$, then
\begin{equation}
\label{constraints-on-solutions-derivative}
\left\{ \begin{array}{l}
\xi_{\pm}(s;t,x) \leq 1, \\
U_{\pm}(s;t,x) \geq 0,\\
V_{\pm}(s;t,x) \geq 0, \end{array} \right. \quad (s,t,x) \in \Gamma_T.
\end{equation}
\label{proposition-derivative}
\end{proposition}

\begin{proof}
Assuming existence of solution $(Z_{\pm},Y_{\pm}) \in C^{1,1,1}(\Gamma_T)$ to
the system of integral equations (\ref{2.11})--(\ref{2.14}), we have
by Proposition \ref{prop-correspondence} and the definition (\ref{introduction-U-V}) that
$(U_{\pm},V_{\pm}) \in C^{0,0,0}(\Gamma_T)$ and $u_{\pm} \in C^{0,0}(\Omega_T)$.
By (\ref{coordinate-1-2-der-sol}), we have $\xi_{\pm}(s;t,x) > 0$ for every $(s,t,x) \in \Gamma_T$.
Then, by using relations (\ref{coordinate-1-2-der-sol-rel}), conditions (\ref{cond3-final}),
and the result of Proposition \ref{proposition-definition}, we obtain
from the system (\ref{der-eq-1}) and (\ref{der-eq-2}) that
$U_{\pm}(s;t,x) \geq 0$ and $V_{\pm}(s;t,x) \geq 0$ for every $(s,t,x) \in \Gamma_T$.
Using relations (\ref{coordinate-1-2-der-sol-rel}) again, we have $\xi_{\pm}(s;t,x) \leq 1$
for every $(s,t,x) \in \Gamma_T$. Thus,
constraints (\ref{constraints-on-solutions-derivative}) have been proved.
\end{proof}

Generally speaking, the chain rule (\ref{correspondence-U-V}) and the representation (\ref{coordinate-1-2-der-sol-rel})
only show that if $U_{\pm}(s;t,x)$ and $V_{\pm}(s;t,x)$ remain bounded and positive for $(s,t,x) \in \Gamma_T$,
then $\xi_{\pm}(s;t,x)$ may still vanish at the same points $(s,t,x) \in \Gamma_T$ for which
either $u_{\pm}(s,\eta_{\pm}(s;t,x))$ or $u_{\mp}(s,\eta_{\pm}(s;t,x))$ become unbounded.
However, divergence of $u_{\pm}(t,x)$ for $(t,x) \in \Omega_T$ contradicts to the result of
Proposition \ref{prop-correspondence}, if the solution $(Z_{\pm},Y_{\pm}) \in C^{1,1,1}(\Gamma_T)$
to the system of integral equations (\ref{2.11})--(\ref{2.14}) is obtained.
Therefore, the essence of the method of an additional argument is to ensure solvability
of the system of integral equations (\ref{2.11})--(\ref{2.14}) in $C^{1,1,1}(\Gamma_T)$,
which would guarantee strict positivity of $\xi_{\pm}(s;t,x)$ for every $(s,t,x) \in \Gamma_T$.

For completeness, we mention that if we substitute (\ref{coordinate-1-2-der-sol}), (\ref{correspondence-U-V})
and (\ref{coordinate-1-2-der-sol-rel}) to the integral equations (\ref{der-eq-1}),
then we obtain
\begin{eqnarray}
\nonumber
u_{\pm}(s,\eta_{\pm}(s;t,x)) & = & \varphi_{\pm}'(\eta_{\pm}(0;t,x)) e^{-\frac{1}{4}
\int_0^s \left[ 3 u_{\pm}(\nu,\eta_{\pm}(\nu;t,x)) + u_{\mp}(\nu,\eta_{\pm}(\nu;t,x)) \right] d\nu} \\
& \phantom{t} &
+ \int_0^s h''(\eta_{\pm}(\nu;t,x)) e^{-\frac{1}{4}
\int_{\nu}^s \left[ 3 u_{\pm}(\tau,\eta_{\pm}(\tau;t,x)) + u_{\mp}(\tau,\eta_{\pm}(\tau;t,x)) \right] d\tau} d\nu,
\label{derivative-1-2-solutions}
\end{eqnarray}
which can be thought as a weak formulation of the system of differential equations
{\small\begin{eqnarray}
\label{derivative-1-2}
\left\{ \begin{array}{l}
\frac{d u_+}{ds}(s,\eta_+(s;t,x)) + \frac{3}{4} u_+^2(s,\eta_+(s;t,x)) + \frac{1}{4} u_+(s,\eta_+(s;t,x))u_-(s,\eta_+(s;t,x)) = h''(\eta_+(s;t,x)), \\
\frac{d u_-}{ds}(s,\eta_-(s;t,x)) + \frac{3}{4} u_-^2(s,\eta_-(s;t,x)) + \frac{1}{4} u_+(s,\eta_-(s;t,x))u_-(s,\eta_-(s;t,x)) = h''(\eta_-(s;t,x)),
\end{array} \right.
\end{eqnarray}}where $0 \leq s \leq t$, subject to the initial conditions $u_{\pm}(0,\eta_{\pm}(0;t,x)) = \varphi_{\pm}'(\eta_{\pm}(0;t,x))$
and the consistency conditions $u_{\mp}(s,\eta_{\pm}(s;t,x)) = u_{\mp}(s,\eta_{\mp}(s;s,\eta_{\pm}(s;s,x)))$.
The differential system (\ref{derivative-1-2}) can be derived
by differentiating system (\ref{sym-system}) with respect to $x$ for appropriate solutions
$z_{\pm} \in C^{2,2}(\Omega_T)$ and using the characteristic equations (\ref{coordinate-1-2}).
Again, control of boundness of $u_{\pm}(t,x)$ for $(t,x) \in \Omega_T$ is very difficult within
the evolution problem (\ref{derivative-1-2}) or the system of integral equations (\ref{derivative-1-2-solutions}).
However, all these difficult steps are avoided in the method of an additional argument.

\section{Local solution to system (\ref{2.11})--(\ref{2.14})}

Here we use the method of Picard's successive approximations to prove existence of a local solution
to the system of integral equations (\ref{2.11})--(\ref{2.14}). At first, we are looking for local solutions
in the space $C^{0,0,0}(\Gamma_T)$. The fixed existence time
$T > 0$ is supposed to be small to ensure that the contraction method works. Then, we obtain local solutions
in the space $C^{1,1,1}(\Gamma_T)$ from the Schauder fixed-point theorem.
Assumptions of both Propositions \ref{proposition-definition} and
\ref{prop-correspondence} are satisfied for the local solutions in $C^{1,1,1}(\Gamma_T)$.
Thus, by correspondence between solutions to the system of integral equations (\ref{2.11})--(\ref{2.14})
and the quasilinear system (\ref{sym-system}), the results obtained in
this section yield the proof of Theorem \ref{theorem-main-local}.

The main difficulty in the proof of existence
of a local solution to the system of integral equations (\ref{2.11})--(\ref{2.14})
in $C^{0,0,0}(\Gamma_T)$ is due to the fact that the integral equation
(\ref{2.14}) is composed of unknown functions. As a result, the method of successive
approximations consists of two levels, similar to what is described in \cite{Alex2}.
In order to close the system of integral equations (\ref{2.11})--(\ref{2.14}) in $\Gamma_T$,
we use the conditions (\ref{cond1}) and (\ref{cond2}) on the function $h$ and initial data $u_0$ and $\eta_0$,
the latter conditions are rewritten for $\varphi_{\pm}$ in the form (\ref{cond-phi-intro-1}).

\begin{lemma}
Assume $h \in C^2_b(\mathbb{R}^+)$ and $\varphi_{\pm} \in C^1_b(\mathbb{R}^+)$
satisfying the constraints (\ref{cond1}) and (\ref{cond-phi-intro-1}). Define
\begin{equation}
\label{constraint-time}
T :=  \min \left(\frac{C_\varphi}{C_h},\frac{1}{15 C_\varphi}\right),
\end{equation}
where $C_{\varphi} := \max\{\| \varphi_+ \|_{C^1_b},\| \varphi_- \|_{C^1_b}\}$ and $C_h := \| h \|_{C^2_b}$.
Then, the system of integral equations (\ref{2.11})--(\ref{2.14}) admits a unique solution in class
$(Z_{\pm},Y_{\pm}) \in C^{0,0,0}({\Gamma}_T)$ such that
\begin{equation}
\label{bound-on-approximations-superior}
\| Z_{\pm} \|, \| Y_{\pm} \| \leq 2 C_{\varphi}.
\end{equation}
\label{lemma-local}
\end{lemma}

\begin{proof}
By Proposition \ref{proposition-definition}, the system of integral equations (\ref{2.11})--(\ref{2.14})
is closed in $\Gamma_T$ in the sense of bounds (\ref{constraints-on-solutions}).
In order to apply the Picard method, we start
with the initial approximations
\begin{equation}
\label{initial-values-first-level}
Z_{\pm (0)}(s;t,x) = Y_{\pm (0)}(s;t,x) = \varphi_{\pm}(x)
\end{equation}
and define the successive approximations $\{ Z_{\pm (n)}, Y_{\pm (n)} \}_{n \in \mathbb{N}}$
from the recursive iterations based on the system of integral equations (\ref{2.11})--(\ref{2.14})
for $n \in \mathbb{N}$:
\begin{eqnarray}
\nonumber
Z_{\pm (n)}(s;t,x) & = & \varphi_{\pm}\left( x - \frac{1}{4} \int_0^t \left[ 3 Z_{\pm (n)}(\nu;t,x) +
Y_{\pm (n)}(\nu;t,x) \right] d \nu \right) \\
\label{2.17} & \phantom{t} & +
\int_0^s h'\left(x - \frac{1}{4} \int_{\nu}^t \left[ 3 Z_{\pm (n)}(\tau;t,x) + Y_{\pm (n)}(\tau;t,x) \right] d \tau \right) d \nu,
\end{eqnarray}
and
\begin{eqnarray}
\label{2.20}
Y_{\pm (n)}(s;t,x) = Z_{\mp (n-1)}\left(s;s,x - \frac{1}{4} \int_s^t \left[ 3 Z_{\pm (n)}(\nu;t,x) + Y_{\pm (n)}(\nu;t,x) \right] d\nu\right).
\end{eqnarray}

The system (\ref{2.17})--(\ref{2.20}) is implicit in $(Z_{\pm (n)},Y_{\pm (n)})$. Therefore, for each $n \in \mathbb{N}$, we obtain
$Z_{\pm (n)},Y_{\pm (n)}$ from another sequence of successive approximations $\{ Z_{\pm (n)}^{(k)}, Y_{\pm (n)}^{(k)} \}_{k \in \mathbb{N}}$
starting with the initial approximations
\begin{equation}
\label{initial-values-second-level}
Z^{(0)}_{\pm (n)}(s;t,x) = Z_{\pm (n-1)}(s;t,x) \quad \mbox{\rm and} \quad Y^{(0)}_{\pm (n)}(s;t,x) = Y_{\pm (n-1)}(s;t,x),
\quad n \in \mathbb{N},
\end{equation}
which is defined at least for $n = 1$. Successive approximations $\{ Z_{\pm (n)}^{(k)}, Y_{\pm (n)}^{(k)} \}_{k \in \mathbb{N}}$
are defined by the explicit iteration scheme for $k \in \mathbb{N}$:
\begin{eqnarray}
\nonumber
Z^{(k)}_{\pm (n)}(s;t,x) & = & \varphi_{\pm}\left(x - \frac{1}{4} \int_0^t
\left[ 3 Z^{(k-1)}_{\pm (n)}(\nu;t,x) + Y^{(k-1)}_{\pm (n)}(\nu;t,x) \right] d \nu \right) \\
\label{2.21} & \phantom{t} & +
\int_0^s h'\left(x - \frac{1}{4} \int_{\nu}^t \left[ 3 Z^{(k-1)}_{\pm (n)}(\tau;t,x) + Y^{(k-1)}_{\pm (n)}(\tau;t,x) \right] d \tau\right) d \nu,
\end{eqnarray}
and
\begin{eqnarray}
Y^{(k)}_{\pm (n)}(s;t,x) = Z_{\mp (n-1)}\left(s;s,x - \frac{1}{4} \int_s^t \left[ 3 Z^{(k-1)}_{\pm (n)}(\nu;t,x) + Y^{(k-1)}_{\pm (n)}(\nu;t,x) \right] d\nu\right).
\label{2.24}
\end{eqnarray}
The construction of successive approximations to the two-level system in $C^{0,0,0}(\Gamma_T)$ is broken into three steps. \\

{\bf Step 1.} We prove for every $n \in \mathbb{N}$ that the sequence
$\{ Z_{\pm (n)}^{(k)}, Y_{\pm (n)}^{(k)} \}_{k \in \mathbb{N}}$ satisfying
(\ref{initial-values-second-level}), (\ref{2.21}), and (\ref{2.24})
converges in $C^{0,0,0}(\Gamma_T)$ for a fixed $T > 0$ satisfying (\ref{constraint-time}), so that we can define
\begin{equation}
\label{limits-second-level}
Z_{\pm (n)}(s;t,x) := \lim_{k \to \infty} Z_{\pm (n)}^{(k)}(s;t,x) \quad \mbox{\rm and} \quad
Y_{\pm (n)}(s;t,x) := \lim_{k \to \infty} Y_{\pm (n)}^{(k)}(s;t,x), \quad n \in \mathbb{N}.
\end{equation}
Let us introduce $C_{\varphi} := \max\{\| \varphi_+ \|_{C^1_b},\| \varphi_- \|_{C^1_b}\}$ and $C_h := \| h \|_{C^2_b}$.
It follows from (\ref{2.21}) and (\ref{2.24}) that
\begin{equation}
\label{bound-on-approximations-uniform}
\| Z^{(k)}_{\pm (n)}\| \leq C_{\varphi} + C_h T \leq 2 C_{\varphi}, \quad
\| Y^{(k)}_{\pm (n)}\| = \| Z_{\pm (n-1)}\|,
\quad k \in \mathbb{N},
\end{equation}
where we have used $C_h T \leq C_{\varphi}$ according to the constraint (\ref{constraint-time}).
Since the bounds (\ref{bound-on-approximations-uniform})
are independent of $k$, if convergence to the limits (\ref{limits-second-level}) can be proved
for each $n \in \mathbb{N}$, then by the induction method, we have
\begin{equation}
\label{bound-on-approximations}
\| Z_{\pm (n)} \|, \| Y_{\pm (n)} \| \leq 2 C_{\varphi}, \quad n \in \mathbb{N}.
\end{equation}
Bounds (\ref{bound-on-approximations}) are also satisfied for $n = 0$.
Now, we establish convergence to the limits in (\ref{limits-second-level}).

By using the fundamental theorem
of calculus and the estimates similar to those in the proof of Proposition \ref{prop-correspondence},
we derive the bounds on the distance between two successive approximations:
\begin{equation}
\label{bounds-second-level}
3 \| Z^{(k+1)}_{\pm (n)} - Z^{(k)}_{\pm (n)} \| + \| Y^{(k+1)}_{\pm (n)} - Y^{(k)}_{\pm (n)}\|
\leq K_{\pm}(T) \left( 3 \| Z^{(k)}_{\pm (n)} - Z^{(k-1)}_{\pm (n)}  \| + \| Y^{(k)}_{\pm (n)} - Y^{(k-1)}_{\pm (n)}  \| \right),
\end{equation}
where we have denoted
\begin{equation}
\label{def-K}
K_{\pm}(T) := \frac{1}{4} \left( 3 C_{\varphi} T + \frac{3}{2} C_h T^2 + \| \partial_x Z_{\mp (n-1)} \| T \right).
\end{equation}

Let us assume by induction that $(Z_{\pm (n-1)},Y_{\pm (n-1)}) \in C^{0,0,1}(\Gamma_T)$ satisfying
\begin{equation}
\label{bound-on-derivatives}
\| \partial_x Z_{\pm (n-1)} \| \le 3 C_\varphi, \quad
\| \partial_x Y_{\pm (n-1)} \| \le 4 C_\varphi, \quad n \in \mathbb{N},
\end{equation}
which is satisfied at least for $n = 1$. It follows from
(\ref{bound-on-approximations}) and (\ref{bounds-second-level}) that
\begin{equation}
\label{bound-on-first-iteration}
\| Z^{(1)}_{\pm (n)} - Z^{(0)}_{\pm (n)} \|, \| Y^{(1)}_{\pm (n)} - Y^{(0)}_{\pm (n)}\| \leq 8 K_{\pm}(T) C_{\varphi}.
\end{equation}
Continuing on with (\ref{bounds-second-level}) and (\ref{bound-on-first-iteration}),
we obtain
\begin{equation}
\label{bound-on-next-iterations}
\| Z^{(k+1)}_{\pm (n)} - Z^{(k)}_{\pm (n)} \|, \| Y^{(k+1)}_{\pm (n)} - Y^{(k)}_{\pm (n)}\| \leq (4 K_{\pm}(T))^k (8 K_{\pm}(T) C_{\varphi}), \quad k \in \mathbb{N}.
\end{equation}
Therefore, the sequence $\{ Z_{\pm (n)}^{(k)}, Y_{\pm (n)}^{(k)} \}_{k \in \mathbb{N}}$ is Cauchy in $C^{0,0,0}(\Gamma_T)$
for each $n \in \mathbb{N}$ if $4 K_{\pm}(T) < 1$. From the definition (\ref{def-K}), bound (\ref{bound-on-derivatives}),
and $C_h T \leq C_{\varphi}$,
we have
$$
4K_{\pm}(T) \leq 3C_{\varphi} T + \frac{3}{2} C_{\varphi} T + 3 C_{\varphi} T = \frac{15}{2} C_{\varphi} T \leq \frac{1}{2},
$$
if $T \leq \frac{1}{15 C_{\varphi}}$, according to the constraint (\ref{constraint-time}).
Hence, for each $n \in \mathbb{N}$, the sequence $\{ Z_{\pm (n)}^{(k)}, Y_{\pm (n)}^{(k)} \}_{k \in \mathbb{N}}$
converges as $k \to \infty$ to a limit denoted by $(Z_{\pm (n)}, Y_{\pm (n)})$ in $C^{0,0,0}(\Gamma_T)$,
as in (\ref{limits-second-level}).

Taking the limit $k \to \infty$ in the recursive system (\ref{2.21})--(\ref{2.24}), we obtain
the recursive system (\ref{2.17})--(\ref{2.20}) for $(Z_{\pm (n)}, Y_{\pm (n)})$ in $C^{0,0,0}(\Gamma_T)$.
Therefore, $(Z_{\pm (n)}, Y_{\pm (n)})$ is a local solution to the system (\ref{2.17})--(\ref{2.20})
for each $n \in \mathbb{N}$ that satisfies bounds (\ref{bound-on-approximations}).
Moreover, from the contraction principle, it follows that the local solution to the system (\ref{2.17})--(\ref{2.20})
is unique in $C^{0,0,0}(\Gamma_T)$ for each $n \in \mathbb{N}$. \\

{\bf Step 2.} We prove that for each $n \in \mathbb{N}$,
the solution $(Z_{\pm (n)}, Y_{\pm (n)}) \in C^{0,0,0}(\Gamma_T)$ to the system of integral equations
(\ref{2.17})--(\ref{2.20}) constructed in Step 1 belongs actually to $C^{0,0,1}(\Gamma_T)$
and satisfies the same bounds (\ref{bound-on-derivatives}) as the previous approximation $(Z_{\pm (n-1)}, Y_{\pm (n-1)})$.
By differentiating the system (\ref{2.17})--(\ref{2.20}) with respect to $x$,
we obtain a system of linear integral equations
\begin{eqnarray}
U_{\pm (n)}(s;t,x) & = &  \varphi'_{\pm}(\cdot)
\left( 1 - \frac{1}{4} \int_0^t (3 U_{\pm (n)}(\nu;t,x) + V_{\pm (n)}(\nu;t,x)) d\nu \right) \nonumber \\
& \phantom{t} & \quad +\int_0^s h''(\cdot)
\left( 1 - \frac{1}{4} \int_{\nu }^t (3 U_{\pm (n)}(\tau;t,x) + V_{\pm (n)}(\tau;t,x)) d\tau \right) d\nu
\label{d2.26}
\end{eqnarray}
and
\begin{eqnarray}
\label{d2.29}
V_{\pm (n)}(s;t,x) = \partial_x Z_{\mp (n-1)}(\cdot)
\left( 1 - \frac{1}{4} \int_0^t (3 U_{\pm (n)}(\nu;t,x) + V_{\pm (n)}(\nu;t,x)) d \nu \right),
\end{eqnarray}
where the arguments of $\varphi_{\pm}'$, $h''$, and $\partial_x Z_{\mp (n-1)}$ are the same as in the system
(\ref{2.17})--(\ref{2.20}). We recall that $\varphi'_{\pm}$, $h''$ are continuous
and by the method of induction, $\partial_x Z_{\mp (n-1)}$ is also taken to be continuous,
for each $n \in \mathbb{N}$. Since  $(Z_{\pm (n)},Y_{\pm (n)}) \in C^{0,0,0}(\Gamma_T)$
is substituted in the arguments of $\varphi_{\pm}'$, $h''$, and $\partial_x Z_{\mp (n-1)}$,
we know that the coefficients of the system of linear integral equations (\ref{d2.26})--(\ref{d2.29})
are all continuous functions in $\Gamma_T$.

We first claim that there exists a unique solution of the system
of linear integral equations (\ref{d2.26})--(\ref{d2.29}) in $C^{0,0,0}(\Gamma_T)$.
Indeed, let us rewrite the system in the form
$$
(I + P) \left[ \begin{array}{c} U_{\pm (n)} \\V_{\pm (n)}\end{array} \right] =
\left[ \begin{array}{c} \varphi'_{\pm}(\cdot) + \int_0^s h''(\cdot) d \nu \\
\partial_x Z_{\mp (n-1)}(\cdot) \end{array} \right],
$$
where $P$ is a perturbation to the identity matrix $I$ given by
$$
P \left[ \begin{array}{c} U_{\pm (n)} \\V_{\pm (n)}\end{array} \right] :=
\frac{1}{4} \left[ \begin{array}{l} \varphi'_{\pm}(\cdot) \int_0^t (3 U_{\pm (n)}(\nu;t,x) + V_{\pm (n)}(\nu;t,x)) d\nu \\
\phantom{texttext} + \int_0^s h''(\cdot) \int_{\nu }^t (3 U_{\pm (n)}(\tau;t,x) + V_{\pm (n)}(\tau;t,x)) d\tau  d \nu \\
\partial_x Z_{\mp (n-1)}(\cdot) \int_0^t (3 U_{\pm (n)}(\nu;t,x) + V_{\pm (n)}(\nu;t,x)) d \nu \end{array} \right].
$$
We estimate the norm of each component of the perturbation $P$ in $C^{0,0,0}(\Gamma_T)$ as follows
\begin{equation}
\label{matrix-P}
\left\| P \left[ \begin{array}{c} U_{\pm (n)} \\V_{\pm (n)}\end{array} \right] \right\|
\leq \frac{1}{4} T C_{\varphi} \left[ \begin{array}{cc} 6 & 2 \\ 9 & 3 \end{array} \right]
\left[ \begin{array}{c} \| U_{\pm (n)} \| \\ \| V_{\pm (n)} \| \end{array} \right],
\end{equation}
where we have used $C_h T \leq C_{\varphi}$ and $\| \partial_x Z_{\mp (n-1)} \| \leq 3 C_{\varphi}$.
Eigenvalues of the matrix in (\ref{matrix-P}) are $0$ and $9$.
If $T C_{\varphi} \leq \frac{1}{15}$, the norm induced by the perturbation $P$ is strictly smaller than one.
Therefore, the matrix integral operator $I + P$ is invertible and a unique solution $(U_{\pm (n)},V_{\pm (n)})$
to the system of linear integral equations (\ref{d2.26})--(\ref{d2.29}) exists in $C^{0,0,0}(\Gamma_T)$.

Next, for every $(s,t,x_0) \in \Gamma_T$, we claim that the quotients
$$
\frac{Z_{\pm (n)}(s;t,x) - Z_{\pm (n)}(s;t,x_0)}{x-x_0} \quad \mbox{\rm and} \quad
\frac{Y_{\pm (n)}(s;t,x) - Y_{\pm (n)}(s;t,x_0)}{x-x_0}
$$
remain bounded as $x \to x_0$ for every $x_0 \in \mathbb{R}^+$. This is shown by
repeating the estimates for the system of integral equations (\ref{2.17})--(\ref{2.20}), where
we are using the constraint on $T$ in (\ref{constraint-time}),
and the smoothness properties on $\varphi_{\pm}$, $h$, and $Z_{\mp (n-1)}$.
Now, by repeating the estimates for bounded functions
$$
E_{\pm (n)}(s;t,x,x_0) := \frac{Z_{\pm (n)}(s;t,x) - Z_{\pm (n)}(s;t,x_0)}{x-x_0} - U_{\pm (n)}(s;t,x_0)
$$
and
$$
F_{\pm (n)}(s;t,x,x_0) := \frac{Y_{\pm (n)}(s;t,x) - Y_{\pm (n)}(s;t,x_0)}{x-x_0} - V_{\pm (n)}(s;t,x_0)
$$
and using uniqueness of solutions of the integral equations (\ref{2.17})--(\ref{2.20})
and their first variations (\ref{d2.26})--(\ref{d2.29}), we obtain for every $(s,t,x_0) \in \Gamma_T$ that
$$
\lim_{x \to x_0} E_{\pm (n)}(s;t,x,x_0) = 0 \quad \mbox{\rm and} \quad \lim_{x \to x_0} F_{\pm (n)}(s;t,x,x_0) = 0.
$$
Therefore, $(Z_{\pm (n)}, Y_{\pm (n)})$
are continuously differentiable with respect to $x$ at every $x_0 \in \mathbb{R}^+$ and
\begin{equation}
\label{equivalence-derivatives-x}
\partial_x Z_{\pm (n)}(s;t,x) = U_{\pm (n)}(s;t,x) \quad \mbox{\rm and} \quad
\partial_x Y_{\pm (n)}(s;t,x) = V_{\pm (n)}(s;t,x), \quad (s,t,x) \in \Gamma_T.
\end{equation}

It remains to verify bounds (\ref{bound-on-derivatives}) for $(Z_{\pm (n)}, Y_{\pm (n)})$.
It follows from the second line of (\ref{matrix-P}) substituted to (\ref{d2.29}) that
\begin{eqnarray}
\nonumber
\| V_{\pm (n)}\| & \leq & \frac{3 C_{\varphi}}{1 - \frac{3}{4} C_{\varphi} T} \left(1 + \frac{3T}{4} \| U_{\pm (n)}\|\right) \\
& \leq & \frac{60 C_{\varphi}}{19} \left( 1 + \frac{3T}{4} \| U_{\pm (n)}\| \right).
\label{bounds-1-Y-der}
\end{eqnarray}
where we have used $C_{\varphi} T \leq \frac{1}{15}$. Substituting this estimate to the first line
of (\ref{matrix-P}) and to equation (\ref{d2.26}) yields
\begin{eqnarray}
\nonumber
\| U_{\pm (n)}\| & \leq & \frac{2 C_{\varphi}}{1 - 3 C_{\varphi} T} \left(1 + C_{\varphi} T \right) \\
& \leq & \frac{5}{2} C_{\varphi} \left(1 + C_{\varphi} T \right) \leq \frac{8}{3} C_{\varphi},
\label{bounds-1-Z-der}
\end{eqnarray}
where we have used again $C_{\varphi} T \leq \frac{1}{15}$. By using
the correspondence (\ref{equivalence-derivatives-x}), we obtain
\begin{equation}
\label{bound-on-derivatives-n}
\| \partial_x Z_{\pm (n)} \| \le 3 C_\varphi, \quad
\| \partial_x Y_{\pm (n)} \| \le 4 C_\varphi, \quad n \in \mathbb{N}.
\end{equation}
The validity of the bounds (\ref{bound-on-derivatives}) for every $n \in \mathbb{N}$ is verified by the induction method. \\

{\bf Step 3.} We prove under the same constraint (\ref{constraint-time}) on $T$
that the sequence $\{ Z_{\pm (n)}, Y_{\pm (n)} \}_{n \in \mathbb{N}}$  defined by
the recursive system (\ref{2.17})--(\ref{2.20}) converges in $C^{0,0,0}(\Gamma_T)$
to the solution $(Z_{\pm}, Y_{\pm}) \in C^{0,0,0}(\Gamma_T)$ satisfying
the system of integral equations (\ref{2.11})--(\ref{2.14}) and bound (\ref{bound-on-approximations-superior}).

After the convergence to the limits (\ref{limits-second-level}) is proved, the index $n$ in
the system of integral equations (\ref{2.17})--(\ref{2.20}) can be incremented by one
using the induction method. Convergence of iterations $\{ Z_{\pm (n)}, Y_{\pm (n)} \}_{n \in \mathbb{N}}$  can be considered
in $C^{0,0,0}(\Gamma_T)$ with standard methods.

It follows from (\ref{2.17}) and (\ref{2.20}) with the fundamental theorem of calculus that
\begin{eqnarray*}
\left\| Z_{\pm (n + 1)} - Z_{\pm (n)} \right\| \leqslant \frac{1}{4}\left( C_\varphi T + \frac{1}{2} C_h T^2 \right)
\left( 3 \left\| Z_{\pm (n + 1)} - Z_{\pm (n)} \right\| + \left\| Y_{\pm (n + 1)} - Y_{\pm (n)} \right\| \right)
\end{eqnarray*}
and
\begin{eqnarray*}
\left\| Y_{\pm (n + 1)} - Y_{\pm (n)} \right\| \leqslant
\frac{1}{4} T \left\| \partial_x Z_{\mp (n)} \right\|
\left( 3 \left\| Z_{\pm (n + 1)} - Z_{\pm (n)} \right\| + \left\| Y_{\pm (n + 1)} - Y_{\pm (n)} \right\| \right) +
\left\| Z_{\mp (n)} - Z_{\mp (n - 1)} \right\|,
\end{eqnarray*}
where $C_{\varphi}$ and $C_{h}$ are the same constants as above. Under the conditions (\ref{constraint-time})
and (\ref{bound-on-derivatives-n}), we obtain
\begin{eqnarray*}
\left\| Z_{\pm (n + 1)} - Z_{\pm (n)} \right\| \leqslant \frac{1}{40}
\left( 3 \left\| Z_{\pm (n + 1)} - Z_{\pm (n)} \right\| + \left\| Y_{\pm (n + 1)} - Y_{\pm (n)} \right\| \right)
\end{eqnarray*}
and
\begin{eqnarray*}
\left\| Y_{\pm (n + 1)} - Y_{\pm (n)} \right\| \leqslant
\frac{1}{20} \left( 3 \left\| Z_{\pm (n + 1)} - Z_{\pm (n)} \right\| + \left\| Y_{\pm (n + 1)} - Y_{\pm (n)} \right\| \right) +
\left\| Z_{\mp (n)} - Z_{\mp (n - 1)} \right\|.
\end{eqnarray*}
From the inequalities above, we obtain
\begin{eqnarray*}
\left\| Z_{\pm (n + 1)} - Z_{\pm (n)} \right\| \leqslant \frac{1}{35} \left\| Z_{\mp (n)} - Z_{\mp (n-1)} \right\|,
\end{eqnarray*}
and hence
\begin{eqnarray*}
\left\| Z_{+ (n + 1)} - Z_{+ (n)} \right\| + \left\| Z_{- (n + 1)} - Z_{- (n)} \right\|
\leqslant \frac{1}{35} \left( \left\| Z_{+ (n)} - Z_{+ (n-1)} \right\| + \left\| Z_{- (n)} - Z_{- (n-1)} \right\| \right).
\end{eqnarray*}
Therefore, the iteration map defined by the system (\ref{2.17})--(\ref{2.20}) is a contraction
in $C^{0,0,0}(\Gamma_T)$. Hence,
the sequence $\{ Z_{\pm (n)}, Y_{\pm (n)} \}_{n \in \mathbb{N}}$ is Cauchy in $C^{0,0,0}(\Gamma_T)$ and
it converges  as $n \to \infty$ to a limit, denoted as $(Z_{\pm},Y_{\pm})$, defined in the same function space.
Moreover, taking the limit $n \to \infty$ in the iterative system (\ref{2.17})--(\ref{2.20}),
we obtain the system of integral equations (\ref{2.11})--(\ref{2.14}) for the limiting functions $(Z_{\pm},Y_{\pm})$.
Therefore, the limiting functions $(Z_{\pm},Y_{\pm})$ are solutions of the system (\ref{2.11})--(\ref{2.14})
in $C^{0,0,0}(\Gamma_T)$. Since the sequence $\{ Z_{\pm (n)}, Y_{\pm (n)} \}_{n \in \mathbb{N}}$ in $C^{0,0,0}(\Gamma_T)$ satisfies
the bounds (\ref{bound-on-approximations}) that are independent of $n$, the limiting functions $(Z_{\pm},Y_{\pm})$ satisfy the same bounds,
which become bounds (\ref{bound-on-approximations-superior}). Finally, it follows from the contraction method
that the local solution $(Z_{\pm},Y_{\pm})$ is unique in $C^{0,0,0}(\Gamma_T)$.
\end{proof}

\begin{lemma}
\label{lemma-implicit}
Under conditions of Lemma \ref{lemma-local}, the unique local solution to
the system of integral equations (\ref{2.11})--(\ref{2.14}) belongs to the class
$(Z_{\pm}, Y_{\pm}) \in C^{0,0,1}({\Gamma}_T)$ and satisfies
\begin{equation}
\label{bound-superior}
\| \partial_x Z_+ \| + \| \partial_x Z_- \|\leq 15 C_{\varphi}, \quad
\| \partial_x Y_+ \| + \| \partial_x Y_- \| \leq 45 C_{\varphi}.
\end{equation}
\end{lemma}

\begin{proof}
First, we prove existence of a unique solution $(U_{\pm},V_{\pm}) \in C^{0,0,0}(\Gamma_T)$
to the integral equations (\ref{der-eq-1})--(\ref{der-eq-2}) under the conditions of Lemma \ref{lemma-local}.
Since solutions for $(Z_{\pm},Y_{\pm}) \in C^{0,0,0}(\Gamma_T)$ are already obtained in Lemma \ref{lemma-local},
the coefficients of the integral equation (\ref{der-eq-1}) and the arguments
of the unknown functions $U_{\mp}$ in (\ref{der-eq-2}) are all continuous functions in $\Gamma_T$.

The first equation (\ref{der-eq-1}) represents a linear relation between $U_{\pm}$ and $V_{\pm}$.
The second equation (\ref{der-eq-2}) is linear with respect to $(V_+,V_-)$ and quadratic with respect to
$(U_+,U_-)$. Therefore, first we solve (\ref{der-eq-2}) to obtain a unique map from
$(U_+,U_-)$ to $(V_+,V_-)$, then we substitute the map to (\ref{der-eq-1})
and solve the system uniquely in $(U_+,U_-)$ by using the Schauder fixed-point theorem.

Let us define a ball in $C^{0,0,0}(\Gamma_T)$ of a finite radius given by
\begin{equation}
\label{ball}
\| U_+ \| + \| U_- \| \leq 15 C_{\varphi} =: \delta.
\end{equation}
The integral equation (\ref{der-eq-2}) is rewritten in the explicit form
\begin{equation}
\label{second-equation}
V_{\pm}(s;t,x) + \frac{1}{4} U_{\mp}(\cdot) \int_s^t V_{\pm}(\nu;t,x) d \nu = F_{\pm} :=
U_{\mp}(\cdot) \left( 1 - \frac{3}{4} \int_s^t U_{\pm}(\nu;t,x) d \nu \right)
\end{equation}
where $U_{\mp}(\cdot)$ refers to
\begin{equation}
\label{arguments-U}
U_{\mp}\left(s;s,x - \frac{1}{4} \int_s^t \left[ 3 Z_{\pm}(\nu;t,x) + Y_{\pm}(\nu;t,x) \right] d \nu \right).
\end{equation}
For every $(U_+,U_-)$ in the ball given by (\ref{ball}), we have
\begin{equation}
\label{estimate-V-U-new}
\left\| \frac{1}{4} U_{\mp}(\cdot) \int_s^t V_{\pm}(\nu;t,x) d \nu \right\| \leq \frac{1}{4} T \| U_{\mp} \| \| V_{\pm} \| \leq
\frac{1}{4} \| V_{\pm} \|,
\end{equation}
where we have used the constraint $C_{\varphi} T \leq \frac{1}{15}$.
Therefore, the second term in (\ref{second-equation}) is strictly smaller than the first term in (\ref{second-equation}).
Inverting the linear operator on $V_{\pm}$ in $C^{0,0,0}(\Gamma_T)$ implies that for every $U_{\pm}$ in the ball given by (\ref{ball}),
there exists a unique solution $V_{\pm} \in C^{0,0,0}(\Gamma_T)$ of equation (\ref{second-equation}) such that
\begin{equation}
\label{ball-on-V}
\| V_{\pm} \| \leq \frac{4}{3} \| F_{\pm} \| \leq \frac{4}{3} \left( 1 + \frac{3}{4} T \| U_{\pm} \| \right) \| U_{\mp} \| \leq
\frac{7}{3} \| U_{\mp} \| \leq 3 \| U_{\mp} \|.
\end{equation}
This unique solution defines a map from $(U_+,U_-) \in C^{0,0,0}(\Gamma_T)$ to $(V_+,V_-) \in C^{0,0,0}(\Gamma_T)$.
Since the integral equation (\ref{second-equation}) is a quadratic polynomial on $(U_+,U_-) \in C^{0,0,0}(\Gamma_T)$,
the map $C^{0,0,0}(\Gamma_T) \ni (U_+,U_-) \mapsto (V_+,V_-) \in C^{0,0,0}(\Gamma_T)$ is $C^{\infty}$ in the ball (\ref{ball}).

Let us estimate the Lipschitz constant for the map $C^{0,0,0}(\Gamma_T) \ni (U_+,U_-) \mapsto (V_+,V_-) \in C^{0,0,0}(\Gamma_T)$.
Denote the values $(V_+',V_-')$ that correspond to the values $(U_+',U_-')$. Note that the arguments of $(U_+',U_-')$
are the same as those of $(U_+,U_-)$ given by (\ref{arguments-U}). Subtracting (\ref{second-equation}) for $(U_+,U_-)$
and $(U_+',U_-')$, we obtain
\begin{eqnarray*}
& V_{\pm} - V_{\pm}' + \frac{1}{4} (U_{\mp} - U_{\mp}') \int_s^t V_{\pm} d \nu +
+ \frac{1}{4} U_{\mp}' \int_s^t (V_{\pm} - V_{\pm}') d \nu\\
& = (U_{\mp} - U_{\pm}') \left( 1 - \frac{3}{4} \int_s^t U_{\pm} d \nu \right)
- \frac{3}{4} U_{\pm}' \int_s^t (U_{\pm} - U_{\pm}') d \nu.
\end{eqnarray*}
Using estimates similar to (\ref{estimate-V-U-new}) and (\ref{ball-on-V}), we obtain
\begin{eqnarray}
\nonumber
\| V_{\pm} - V_{\pm}' \| & \leq &  \frac{4}{3} \left( 1 + \frac{3}{4} T \| U_{\pm} \| + \frac{1}{4} T \| V_{\pm} \| \right) \| U_{\mp} - U_{\mp}' \|
+ T \| U_{\mp}' \| \| U_{\pm} - U_{\pm}' \| \\
& \leq &  \frac{7}{3} \| U_{\mp} - U_{\mp}' \| + \| U_{\pm} - U_{\pm}' \| \leq 3 \| U_{\mp} - U_{\mp}' \| + \| U_{\pm} - U_{\pm}' \|.
\label{map-Lipschitz}
\end{eqnarray}

Next, we substitute the map $C^{0,0,0}(\Gamma_T) \ni (U_+,U_-) \mapsto (V_+,V_-) \in C^{0,0,0}(\Gamma_T)$
to the integral equation (\ref{der-eq-1}) and rewrite it in the explicit form:
\begin{eqnarray}
\nonumber
& U_{\pm}(s;t,x) + \frac{1}{4} \varphi_{\pm}'(\cdot)
\int_0^t (3 U_{\pm}(\nu;t,x) + V_{\pm}(\nu;t,x)) d\nu \\
& \phantom{text} + \frac{1}{4} \int_0^s h''\left(\cdot\right)  \int_{\nu}^t \left[ 3 U_{\pm}(\tau;t,x) + V_{\pm}(\tau;t,x) \right] d\tau  d\nu =  G_{\pm} := \varphi'_{\pm}(\cdot) + \int_0^s h''\left(\cdot\right) d\nu,
\label{first-equation}
\end{eqnarray}
where the arguments for $\varphi_{\pm}'$ and $h''$ are uniquely defined
continuous functions in $\Gamma_T$. Since the mapping $C^{0,0,0}(\Gamma_T) \ni (U_+,U_-) \mapsto (V_+,V_-) \in C^{0,0,0}(\Gamma_T)$
is nonlinear, we solve the system of two integral equations (\ref{first-equation}) by using the Schauder fixed-point theorem
in the ball (\ref{ball}). By using bounds (\ref{ball-on-V}) and the constraint $C_h T \leq C_{\varphi}$,
we estimate the integral terms in the left-hand-side of system (\ref{first-equation}) as follows:
$$
\left\| \frac{1}{4} \varphi_{\pm}'(\cdot)
\int_0^t (3 U_{\pm}(\nu;t,x) + V_{\pm}(\nu;t,x)) d\nu \right\| \leq
\frac{1}{4} T C_{\varphi} (3 \| U_{\pm} \| + \| V_{\pm} \|) \leq
\frac{1}{20} \left( \| U_+ \| + \| U_- \| \right)
$$
and
$$
\left\| \frac{1}{4} \int_0^s h''\left(\cdot\right)  \int_{\nu}^t \left[ 3 U_{\pm}(\tau;t,x) + V_{\pm}(\tau;t,x) \right] d\tau  d\nu  \right\| \leq
\frac{1}{4} T^2 C_h (3 \| U_{\pm} \| + \| V_{\pm} \|) \leq
\frac{1}{20} \left( \| U_+ \| + \| U_- \| \right),
$$
where we have used the constraint $T C_{\varphi} \leq \frac{1}{15}$. The integral terms in system (\ref{first-equation})
are strictly smaller than the identity terms in the ball (\ref{ball}). Therefore,
writing the fixed-point problem in the form
\begin{equation}
\label{family-problems}
\left[ \begin{array}{c} U_+ \\ U_- \end{array} \right] = \left[ \begin{array}{c} G_+ \\ G_- \end{array} \right] +
\mathcal{T} \left[ \begin{array}{c} U_+ \\ U_- \end{array} \right]
\end{equation}
shows that the nonlinear integral operator $\mathcal{T}$ maps the ball (\ref{ball}) to its smaller subset.
The inhomogeneous terms $G_{\pm}$ given by (\ref{first-equation}) are bounded by $\| G_{\pm} \| \leq 2 C_{\varphi}$.
By the Schauder fixed-point theorem, there exists a solution $(U_+,U_-) \in C^{0,0,0}(\Gamma_T)$
to the fixed-point problem (\ref{family-problems}) in the ball (\ref{ball}). The solution to
the system of integral equations (\ref{first-equation}) satisfies the bound
$$
\| U_+ \| + \| U_- \| \leq \frac{4 C_{\varphi}}{1 - 3 TC_{\varphi}/2} \leq \frac{40}{9} C_{\varphi} < \delta
$$
and hence belongs to the ball (\ref{ball}). The solution is unique if the operator $\mathcal{T}$ is a contraction
in the ball (\ref{ball}) \cite{Kellogg}. This is proved directly by using the Lipschitz continuity of the
map $C^{0,0,0}(\Gamma_T) \ni (U_+,U_-) \mapsto (V_+,V_-) \in C^{0,0,0}(\Gamma_T)$ with the Lipschitz
constant given by (\ref{map-Lipschitz}). Indeed, we have
\begin{eqnarray*}
\left\| \frac{1}{4} \varphi_{\pm}'
\int_0^t [3 (U_{\pm} - U_{\pm}') + (V_{\pm} - V_{\pm}') ] d\nu \right\| & \leq &
 \frac{1}{4} T C_{\varphi} (3 \| U_{\pm} - U_{\pm}' \| + \| V_{\pm} - V_{\pm}' \|)\\
 & \leq & \frac{1}{15} \| U_{\pm} - U_{\pm}' \| + \frac{1}{20} \| U_{\mp} - U_{\mp}' \|
\end{eqnarray*}
and a similar estimate for the second term in $T$. Therefore, the operator $\mathcal{T}$ is a contraction
in the ball (\ref{ball}) so that the solution $(U_+,U_-) \in C^{0,0,0}(\Gamma_T)$ is unique.

For every $(s,t,x_0) \in \Gamma_T$, we repeat the estimates for the quotients
$$
\frac{Z_{\pm}(s;t,x) - Z_{\pm}(s;t,x_0)}{x-x_0} \quad \mbox{\rm and} \quad
\frac{Y_{\pm}(s;t,x) - Y_{\pm}(s;t,x_0)}{x-x_0}
$$
and prove that they remain bounded as $x \to x_0$ for every $x_0 \in \mathbb{R}^+$.
Furthermore, by repeating the estimates for bounded functions
$$
E_{\pm}(s;t,x,x_0) := \frac{Z_{\pm}(s;t,x) - Z_{\pm}(s;t,x_0)}{x-x_0} - U_{\pm}(s;t,x_0)
$$
and
$$
F_{\pm}(s;t,x,x_0) := \frac{Y_{\pm}(s;t,x) - Y_{\pm}(s;t,x_0)}{x-x_0} - V_{\pm}(s;t,x_0)
$$
and using uniqueness of solutions of the integral equations (\ref{2.11})--(\ref{2.14})
and their first variations (\ref{der-eq-1})--(\ref{der-eq-2}), we obtain for every $(s,t,x_0) \in \Gamma_T$ that
$$
\lim_{x \to x_0} E_{\pm}(s;t,x,x_0) = 0 \quad \mbox{\rm and} \quad \lim_{x \to x_0} F_{\pm}(s;t,x,x_0) = 0.
$$
Therefore, $(Z_{\pm}, Y_{\pm})$ are continuously differentiable with respect to $x$ at every $x_0 \in \mathbb{R}^+$
and the correspondence (\ref{introduction-U-V})
is established. Bounds (\ref{bound-superior})
follow from bounds (\ref{ball}) and (\ref{ball-on-V}).
\end{proof}

\begin{remark}\label{remark-invertibility}
Bounds (\ref{bound-superior}) are bigger than the $n$-independent bounds (\ref{bound-on-derivatives-n}).
Nevertheless, the bigger bounds (\ref{bound-superior}) are still sufficient for invertibility of the characteristic
coordinates $\xi_{\pm}(s;t,x)$ with respect to $x$ for every $(s,t,x) \in \Gamma_T$. Indeed,
bounds (\ref{bound-superior}) imply that
$$
\left\| \frac{1}{4} \int_s^t \left[ 3 U_{\pm}(\nu;t,x) + V_{\pm}(\nu;t,x) \right] d \nu \right\|
\leq \frac{1}{4} T (3 \| U_{\pm} \| + \| V_{\pm} \|) \leq \frac{3}{4} T (\| U_+ \| + \| U_- \|) \leq\frac{3}{4},
$$
where the constraint $C_{\varphi} T \leq \frac{1}{15}$ has been used.
Therefore, it follows from (\ref{coordinate-1-2-der-sol-rel}) that
if $(U_{\pm},V_{\pm})$ are $x$-derivatives of the local solution $(Z_{\pm},Y_{\pm})$
in Lemmas \ref{lemma-local} and \ref{lemma-implicit}, then
$\xi_{\pm}(s;t,x) > 0$ for every $(s,t,x) \in \Gamma_T$.
\end{remark}

\begin{lemma}
Under conditions of Lemma \ref{lemma-local}, the unique local solution to
the system of integral equations (\ref{2.11})--(\ref{2.14}) belongs to the class
$(Z_{\pm}, Y_{\pm}) \in C^{1,1,1}({\Gamma}_T)$.
\label{lemma-local-improved}
\end{lemma}

\begin{proof}
By Lemmas \ref{lemma-local} and \ref{lemma-implicit},
there exists a unique solution $(Z_{\pm}, Y_{\pm}) \in C^{0,0,1}({\Gamma}_T)$
to the system of integral equations (\ref{2.11})--(\ref{2.14}). We show
that the solution actually belongs to $C^{1,1,1}(\Gamma_T)$.

Let us compute the derivatives of the system of integral equations (\ref{2.11})--(\ref{2.14}) in $t$:
\begin{eqnarray}
\partial_t Z_{\pm}(s;t,x) & = & -\frac{1}{4} \varphi'_{\pm}(\cdot) \left( 3 Z_{\pm}(t;t,x) + Y_{\pm}(t;t,x) \right)
-\frac{1}{4} \int_0^s h''(\cdot) d \nu \left( 3 Z_{\pm}(t;t,x) + Y_{\pm}(t;t,x) \right) \nonumber \\
& \phantom{t} & -\frac{1}{4} \varphi'_{\pm}(\cdot) \int_0^t (3 \partial_t Z_{\pm}(\nu;t,x) + \partial_t Y_{\pm}(\nu;t,x)) d\nu \nonumber \\
& \phantom{t} & -\frac{1}{4} \int_0^s h''(\cdot) \left( \int_{\nu }^t (3 \partial_t Z_{\pm}(\tau;t,x) + \partial_t Y_{\pm}(\tau;t,x)) d\tau \right) d\nu
\label{time-der-1}
\end{eqnarray}
and
\begin{eqnarray}
\label{time-der-2}
\partial_t Y_{\pm}(s;t,x) = -\frac{1}{4} \partial_x Z_{\mp}(\cdot)
\int_0^t (3 \partial_t Z_{\pm}(\nu;t,x) + \partial_t Y_{\pm}(\nu;t,x)) d \nu,
\end{eqnarray}
where the arguments of $\varphi_{\pm}'$, $h''$, and $\partial_x Z_{\pm}$ are the same as in the system (\ref{2.11})--(\ref{2.14}).
They are given continuous functions of their arguments in the linear integral equations (\ref{time-der-1})--(\ref{time-der-2}).

Using similar estimates as in Step 2 in the proof of Lemma \ref{lemma-local}, we can use invertibility of
the linear integral operators
and prove existence and uniqueness of solutions to the system (\ref{time-der-1})--(\ref{time-der-2})
for $(\partial_t Z_{\pm},\partial_t Y_{\pm})$ in $C^{0,0,0}(\Gamma_T)$. Moreover,
the $t$-derivatives of $(Z_{\pm},Y_{\pm})$ satisfy the following bounds:
$$
\| \partial_t Z_{\pm} \| \leq \frac{1}{4} \left(C_{\varphi} T + \frac{1}{2} C_h T^2 \right)
\left( 3 \| \partial_t Z_{\pm} \|  + \| \partial_t Y_{\pm} \| \right) + \frac{1}{4} (C_{\varphi} + C_h T)
\left( 3 \| Z_{\pm} \|  + \| Y_{\pm} \| \right)
$$
and
$$
\| \partial_t Y_{\pm} \| \leq \frac{1}{4} \| \partial_x Z_{\mp} \| T
\left( 3 \| \partial_t Z_{\pm} \|  + \| \partial_t Y_{\pm} \| \right).
$$
By using bounds (\ref{constraint-time}), (\ref{bound-on-approximations-superior}), and (\ref{bound-superior}),
we confirm that $\| \partial_t Z_{\pm} \|$ and $\| \partial_t Y_{\pm} \|$ are bounded in $\Gamma_T$.
Therefore, the solution $(Z_{\pm},Y_{\pm})$ to the system of integral equations (\ref{2.11})--(\ref{2.14})
belongs to $C^{0,1,1}(\Gamma_T)$.

Finally, we compute the derivatives of the system of integral equations (\ref{2.11})--(\ref{2.14}) in $s$:
\begin{eqnarray}
\partial_s Z_{\pm}(s;t,x) & = & h'\left(x - \frac{1}{4} \int_{s}^t (3 Z_{\pm}(\nu;t,x) + Y_{\pm}(\nu;t,x)) d \nu \right)
\label{time-der-3}
\end{eqnarray}
and
\begin{eqnarray}
\label{time-der-4}
\partial_s Y_{\pm}(s;t,x) = \partial_s Z_{\mp}(\cdot) + \partial_t Z_{\mp}(\cdot) +
\frac{1}{4} \partial_x Z_{\mp}(\cdot) (3 Z_{\pm}(s;t,x) + Y_{\pm}(s;t,x)).
\end{eqnarray}
From (\ref{time-der-3}), we confirm that $\| \partial_s Z_{\pm} \|$ is bounded in $\Gamma_T$. Then, from
(\ref{time-der-4}) and the bounds on $Z_{\pm} \in C^{1,1,1}(\Gamma_T)$, we confirm that
$\| \partial_s Y_{\pm} \|$ is also bounded in $\Gamma_T$.
Therefore, the solution $(Z_{\pm},Y_{\pm})$ to the system of integral equations (\ref{2.11})--(\ref{2.14})
belongs to $C^{1,1,1}(\Gamma_T)$.
\end{proof}

The proof of Theorem \ref{theorem-main-local} follows from the results of Lemmas \ref{lemma-local}, \ref{lemma-implicit},
and \ref{lemma-local-improved}, as well as the correspondence result of Proposition \ref{prop-correspondence}.
Solutions to the shallow-water system (\ref{shallow-water}) are related to the solutions to the system (\ref{sym-system})
by using the transformation (\ref{RH-inverse}).

\section{Global solution to system (\ref{2.11})--(\ref{2.14})}

It follows from the correspondence $z_{\pm}(t,x) = Z_{\pm}(t;t,x)$ for $(t,x) \in
\Omega_T$ and the bounds (\ref{bound-on-approximations-superior}) and (\ref{bound-superior}) that
the local solution to the system (\ref{sym-system}) at time $t = T$
satisfies the estimates
\begin{equation}
\label{bound-first-iteration}
\| z_{\pm}(T,\cdot) \|_{C^1_b} \leq 15 C_{\varphi}.
\end{equation}
If we attempt to continue this local solution beyond the time $t = T$
by a recurrent use of Lemmas \ref{lemma-local}, \ref{lemma-implicit}, and \ref{lemma-local-improved},
then we will run into the following obstacle.

Let us denote the solution to the system of integral equations (\ref{2.11})--(\ref{2.14})
given by  Lemmas \ref{lemma-local}, \ref{lemma-implicit}, and \ref{lemma-local-improved} extended from time
$T_{m-1}$ to $T_m$ by $(Z_{\pm}^{(m)},Y_{\pm}^{(m)})$ for $m \in \mathbb{N}$, where $T_0 = 0$.
Then, iterating bound (\ref{bound-first-iteration})
with the bounds (\ref{bound-on-approximations-superior}) and (\ref{bound-superior}), we obtain
\begin{equation}
\label{bound-next-iterations}
\| z_{\pm}^{(m)}(T_m,\cdot) \|_{C^1_b} \leq 15^m C_{\varphi}, \quad m \in \mathbb{N}.
\end{equation}
Furthermore, using the constraint (\ref{constraint-time}) on the continuation time, we have
\begin{equation}
\label{bound-time-next-iterations}
T_m - T_{m-1} \leq \frac{1}{15^{m+1} C_{\varphi}}, \quad m \in \mathbb{N}.
\end{equation}
Since the series $\sum_{m \in \mathbb{N}} 15^{-m}$ converges, we have
$T_{\infty} := \lim_{m \to \infty} T_m < \infty$,
so that the continuation technique results in a local solution to the system (\ref{sym-system})
over a finite time span $[0,T_{\infty})$.

In order to be able to extend the local solution to the system of integral equations (\ref{2.11})--(\ref{2.14})
without restriction on time $T$, we shall find a sharper bounds on the growth
of the $x$-derivatives of the solution $(Z_{\pm},Y_{\pm})$. This is only possible
under additional conditions (\ref{cond3-on-h}) and (\ref{cond3}) on the function $h$ and initial data,
the latter conditions are rewritten in the form (\ref{cond-phi-intro-2}).
The key result is the following lemma.

\begin{lemma}
\label{lemma-bound-on-derivative}
In addition to the conditions of Lemma \ref{lemma-local}, assume that conditions (\ref{cond3-on-h}) and
(\ref{cond-phi-intro-2}) are satisfied.
Then, the unique solution $(Z_{\pm}, Y_{\pm}) \in C^{0,0,1}({\Gamma}_T)$
to the system of integral equations (\ref{2.11})--(\ref{2.14}) constructed in Lemmas \ref{lemma-local} and \ref{lemma-implicit}
satisfy the improved bounds
\begin{equation}
\label{bound-on-approximations-improved}
\| \partial_x Z_{\pm} \|, \| \partial_x Y_{\pm} \| \leq 2 C_{\varphi}.
\end{equation}
\end{lemma}

\begin{proof}
The components $(U_{\pm},V_{\pm})$ satisfy the system of integral
equations (\ref{der-eq-1})-(\ref{der-eq-2}) with the correspondence (\ref{introduction-U-V}).
By Proposition \ref{proposition-derivative} and Remark \ref{remark-invertibility}, we have
$0 < \xi_{\pm}(s;t,x) \leq 1$, $U_{\pm}(s;t,x) \geq 0$, and $V_{\pm}(s;t,x) \geq 0$
for every $(s,t,x) \in \Gamma_T$, where $\xi_{\pm}$ are related to $U_{\pm}$ and $V_{\pm}$
by (\ref{coordinate-1-2-der-sol-rel}). Therefore,
the integral equations (\ref{der-eq-1})--(\ref{der-eq-2}) imply the bounds
$$
\| U_{\pm} \| \leq C_{\varphi} + C_h T \leq 2 C_{\varphi}, \quad \| V_{\pm} \| \leq \| U_{\mp} \| \leq 2 C_{\varphi},
$$
where we have used $C_h T \leq C_{\varphi}$ as in Lemma \ref{lemma-local}.
Due to the correspondence (\ref{introduction-U-V}), we have obtained the bounds
(\ref{bound-on-approximations-improved})
\end{proof}

The sharper bounds (\ref{bound-on-approximations-improved}) can be used to continue the local solution
$z_{\pm}(t,x) = Z_{\pm}(t;t,x)$ to the system (\ref{sym-system}) globally in time. The next lemma establish piecewise continuation
of solutions to the system of integral equations (\ref{2.11})--(\ref{2.14}) in $C^{1,1,1}(\Gamma_T)$ for larger values of $T$.

\begin{lemma}
Let $(Z_{\pm}^{(m)},Y_{\pm}^{(m)})$ for $m \in \mathbb{N}$ denote the sequence of solutions
to the system of integral equations (\ref{2.11})--(\ref{2.14}) on the interval $[T_{m-1},T_m]$
starting with initial data
$$
z_{\pm}(T_{m-1},x) = Z_{\pm}^{(m-1)}(T_{m-1};T_{m-1},x),
$$
where $T_0 = 0$ and $Z_{\pm}^{(0)}(0;0,x) = \varphi_{\pm}(x)$.
Assume $h \in C^2_b(\mathbb{R}^+)$ and $\varphi_{\pm} \in C^1_b(\mathbb{R}^+)$
satisfy the bounds (\ref{cond1}), (\ref{cond3-on-h}), (\ref{cond-phi-intro-1}),  and (\ref{cond-phi-intro-2}).
Define $C_{\varphi} := \max\{\| \varphi_+ \|_{C^1_b},\| \varphi_- \|_{C^1_b}\}$ and $C_h := \| h \|_{C^2_b}$.
Assume that $(Z_{\pm}^{(m)},Y_{\pm}^{(m)}) \in C^{1,1,1}(\Gamma_{T_m-T_{m-1}})$ for an $m \in \mathbb{N}$
satisfies the bounds
\begin{equation}
\label{bound-on-approximations-before}
\| Z_{\pm}^{(m)} \|, \| Y_{\pm}^{(m)} \|, \| \partial_x Z_{\pm}^{(m)} \|, \| \partial_x Y_{\pm}^{(m)} \| \leq (m+1) C_{\varphi}.
\end{equation}
Then, the system of integral equations (\ref{2.11})--(\ref{2.14}) admits a unique solution in class
$$
(Z_{\pm}^{(m+1)}, Y_{\pm}^{(m+1)}) \in C^{1,1,1}({\Gamma}_{T_{m+1}-T_m})
$$
satisfying the bounds
\begin{equation}
\label{bound-on-approximations-after}
\| Z_{\pm}^{(m+1)} \|, \| Y_{\pm}^{(m+1)} \|, \| \partial_x Z_{\pm}^{(m+1)} \|,  \| \partial_x Y_{\pm}^{(m+1)} \| \leq (m+2) C_{\varphi},
\end{equation}
while the time span $[T_m,T_{m+1}]$ is defined by
\begin{equation}
\label{constraint-time-after}
T_{m+1}-T_m :=  \min \left(\frac{C_\varphi}{C_h},\frac{1}{15 (m+1) C_\varphi}\right).
\end{equation}
\label{lemma-continuation}
\end{lemma}

\begin{proof}
The first step of the induction method with bound (\ref{bound-on-approximations-before}) and
the time constraint (\ref{constraint-time-after}) is justified by Lemmas \ref{lemma-local}, \ref{lemma-implicit},
\ref{lemma-local-improved}, and \ref{lemma-bound-on-derivative}.

By Proposition \ref{proposition-definition}, the system of integral equations  (\ref{2.11})--(\ref{2.14})  remains closed in
${\Gamma}_{T_m-T_{m-1}}$, so that $z_{\pm}(T_{m},x) \leq 0$ and $\partial_x z_{\pm}(T_m,x) \geq 0$ remain true
for every $x \in \mathbb{R}^+$. Then, the system of integral equations (\ref{2.11})--(\ref{2.14})
remains closed in $\Gamma_{T_{m+1}-T_m}$ as long as the solution exists. Let us denote $T := T_{m+1}-T_m$.

We review bounds used in the proof of Lemma \ref{lemma-local}. Since the superscript now refer
to the $(m+1)$-th local solution defined on the interval $[T_{m},T_{m+1}]$, we only look at the convergence
of iterations defined by the system of implicit integral equations (\ref{2.17})--(\ref{2.20}).
It follows from these integral equations that bounds (\ref{bound-on-approximations})
for the successive approximations $\{ Z_{\pm (n)}^{(m+1)}, Y_{\pm (n)}^{(m+1)} \}_{n \in \mathbb{N}}$
become
\begin{equation}
\label{bound-on-approximations-after-proof} \left\{ \begin{array}{l}
\| Z^{(m+1)}_{\pm (n)}\| \leq (m+1) C_{\varphi} + C_h T \leq (m+2) C_{\varphi}, \\
\| Y^{(m+1)}_{\pm (n)}\| = \| Z^{(m+1)}_{\pm (n-1)}\| \leq (m+2) C_{\varphi}, \end{array} \right.
\quad n \in \mathbb{N},
\end{equation}
where we have used $C_h T \leq C_{\varphi}$ according to the constraint (\ref{constraint-time-after}).
If convergence of successive approximations $\{ Z_{\pm (n)}^{(m+1)}, Y_{\pm (n)}^{(m+1)}  \}_{n \in \mathbb{N}}$
as $n \to \infty$ is proved, then bounds (\ref{bound-on-approximations-after-proof})
yield the first bounds in (\ref{bound-on-approximations-after}).
To prove the convergence, we first assume as in Step 1 that
\begin{equation}
\label{bound-on-derivatives-proof}
\| \partial_x Z^{(m+1)}_{\pm (n-1)} \| \le (2 m+3) C_\varphi, \quad
\| \partial_x Y^{(m+1)}_{\pm (n-1)} \| \le (3 m+4) C_\varphi, \quad n \in \mathbb{N},
\end{equation}
which is true for $n = 1$.
From the definition (\ref{def-K}), bounds (\ref{bound-on-approximations-before}), (\ref{bound-on-derivatives-proof}),
and $C_h T \leq C_{\varphi}$, convergence of
successive approximations at the second level of Picard iterations
in $C^{0,0,0}(\Gamma_{T_{m+1}-T_m})$ (Step 1) is guaranteed if
\begin{equation}
4 K_{\pm}(T) \leq 3 (m+1) C_{\varphi} T + \frac{3}{2} C_h T^2 + (2 m+3) C_{\varphi} T
\leq \frac{5 (2m+3)}{2} C_{\varphi} T \leq \frac{2m +3}{6 (m+1)} < 1,
\end{equation}
where we have used $C_{\varphi} T \leq \frac{1}{15 (m+1)}$ as in the constraint (\ref{constraint-time-after}).
Thus, successive approximations at the second level of Picard iterations
converge in $C^{0,0,0}(\Gamma_{T})$ to
the solution $\{ Z_{\pm (n)}^{(m+1)}, Y_{\pm (n)}^{(m+1)} \}_{n \in \mathbb{N}}$ for every $n \in \mathbb{N}$.

We hence check that $\{ Z_{\pm (n)}^{(m+1)}, Y_{\pm (n)}^{(m+1)} \}_{n \in \mathbb{N}}$ belongs to $C^{0,0,1}(\Gamma_T)$
(Step 2). Let us now rewrite bounds (\ref{matrix-P}) in order to check consistency with the bounds (\ref{bound-on-derivatives-proof}).
We obtain
\begin{equation}
\label{matrix-P-further}
\left\| P \left[ \begin{array}{c} U_{\pm (n)} \\V_{\pm (n)}\end{array} \right] \right\|
\leq \frac{1}{4} T C_{\varphi} \left[ \begin{array}{cc} 3(m+2) & (m+2) \\ 3(2m+3) & (2m+3) \end{array} \right]
\left[ \begin{array}{c} \| U_{\pm (n)} \| \\ \| V_{\pm (n)} \| \end{array} \right],
\end{equation}
where we have used $C_h T \leq C_{\varphi}$ and $\|  \partial_x Z^{(m+1)}_{\pm (n-1)} \| \le (2 m+3) C_\varphi$.
Since the upper bound in (\ref{matrix-P-further}) has the norm being strictly smaller than one, under
the constraint (\ref{constraint-time-after}) on the time step $T$, we establish existence and uniqueness
of partial derivatives of  $(Z_{\pm (n)}^{(m+1)}, Y_{\pm (n)}^{(m+1)})$ in $x$ for each $n \in \mathbb{N}$.
Moreover, we can estimate them by obtaining bounds similar to (\ref{bounds-1-Y-der}) and (\ref{bounds-1-Z-der}).
By using (\ref{matrix-P-further}), we obtain
\begin{eqnarray}
\nonumber
\| \partial_x Y^{(m+1)}_{\pm (n)}\| & \leq & \frac{(2 m+3) C_{\varphi}}{1 - \frac{2 m+3}{4} C_{\varphi} T}
\left(1 + \frac{3T}{4} \| \partial_x Z_{\pm (n)}^{(m+1)} \|\right) \\
\label{Y-bounds}
& \leq & \frac{20 (2 m+3) C_{\varphi}}{19} \left( 1 + \frac{3T}{4} \| \partial_x Z^{(m+1)}_{\pm (n)}\| \right),
\end{eqnarray}
where we have used (\ref{constraint-time-after}) as well as $2m+3 \leq 3(m+1)$.
By using (\ref{constraint-time-after}), (\ref{matrix-P-further}), (\ref{Y-bounds}),
$2m+3 \leq 3(m+1)$, and $m+2 \leq 2(m+1)$, we obtain
\begin{eqnarray}
\nonumber
\| \partial_x Z^{(m+1)}_{\pm (n)}\| & \leq &
\frac{(m+2) C_{\varphi}}{1 - \frac{15 (m+2)}{19} C_{\varphi} T} \left(1 + \frac{5 (2 m+3)}{19} C_{\varphi} T \right) \\
& \leq & \frac{20}{17} (m+2) C_{\varphi} \leq (2m+3) C_{\varphi}. \label{Z-bounds}
\end{eqnarray}
Substituting (\ref{Z-bounds}) to (\ref{Y-bounds}), we
obtain
\begin{eqnarray}
\nonumber
\| \partial_x Y^{(m+1)}_{\pm (n)}\|
& \leq & \frac{20 (2m+3) C_{\varphi}}{19} \left( 1 + \frac{3 (2m+3)}{4} C_{\varphi} T \right) \\
& \leq & \frac{23 (2m+3)}{19} C_{\varphi} \leq (3m+4) C_{\varphi}.
\label{YY-bounds}
\end{eqnarray}
By the induction method, we obtain that bounds (\ref{bound-on-derivatives-proof})
are valid for every $n \in \mathbb{N}$.

Convergence of the successive approximations $\{ Z_{\pm (n)}^{(m+1)}, Y_{\pm (n)}^{(m+1)} \}_{n \in \mathbb{N}}$  at the first level of
Picard iterations is proved in $C^{0,0,0}(\Gamma_T)$ similarly to the proof of Lemma \ref{lemma-local} (Step 3).
Since the sequence $\{ Z_{\pm (n)}^{(m+1)}, Y_{\pm (n)}^{(m+1)} \}_{n \in \mathbb{N}}$ satisfies
the bounds (\ref{bound-on-approximations-after-proof}) that are independent of $n$,
the limiting functions $(Z_{\pm}^{(m+1)},Y_{\pm}^{(m+1)}) \in C^{0,0,0}(\Gamma_T)$
satisfy the first two bounds in (\ref{bound-on-approximations-after}).

Although the bounds (\ref{bound-on-derivatives-proof}) are independent of $n$, we still need to
prove that $(Z_{\pm}^{(m+1)},Y_{\pm}^{(m+1)})$ belong to $C^{0,0,1}(\Gamma_T)$.
We hence follow the proof of Lemma \ref{lemma-implicit} and obtain
$(Z_{\pm}^{(m+1)},Y_{\pm}^{(m+1)}) \in C^{0,0,1}(\Gamma_T)$
together with the bounds
\begin{equation}
\label{bound-on-derivatives-proof-non-sharp}
\| \partial_x Z^{(m+1)}_+ \| + \| \partial_x Z^{(m+1)}_- \| \le 15 (m+1) C_\varphi, \quad
\| \partial_x Y^{(m+1)}_+ \| + \| \partial_x Y^{(m+1)}_- \|\le 45 (m+1) C_\varphi.
\end{equation}
Although the bounds (\ref{bound-on-derivatives-proof-non-sharp}) are bigger than bounds
(\ref{bound-on-derivatives-proof}), which are independent of $n$,
they are sufficient to control the local solution $(Z_{\pm}^{(m+1)},Y_{\pm}^{(m+1)})$ on $\Gamma_T$.
In particular, the characteristic coordinates are still invertible in $x$, because
the integral part of (\ref{coordinate-1-2-der-sol-rel}) is estimated as follows:
$$
\frac{1}{4} T \left( 3 \| \partial_x Z_{\pm}^{(m+1)} \| + \| \partial_x Y_{\pm}^{(m+1)} \| \right) \leq
\frac{3}{4} 15 (m+1) C_{\varphi} T \leq \frac{3}{4}.
$$
As a result, for the local solution in $(Z_{\pm}^{(m+1)},Y_{\pm}^{(m+1)}) \in C^{0,0,1}(\Gamma_T)$,
we still have $\xi_{\pm}(s;t,x) > 0$ for every $(s,t,x) \in \Gamma_T$.

The proof of Lemma \ref{lemma-local-improved} applies verbatim, so that
we actually have  $(Z_{\pm}^{(m+1)},Y_{\pm}^{(m+1)}) \in C^{1,1,1}(\Gamma_T)$.

Finally, we improve the bounds (\ref{bound-on-derivatives-proof-non-sharp})
by using the technique in Lemma \ref{lemma-bound-on-derivative}. In particular,
we have $\partial_x Z_{\pm}^{(m+1)}(s;t,x) \geq 0$ and $\partial_x Y_{\pm}^{(m+1)}(s;t,x) \geq 0$,
and $\xi_{\pm}(s;t,x) \leq 1$ for every $(s,t,x) \in \Gamma_T$. As a result,
the integral equations (\ref{der-eq-1})--(\ref{der-eq-2}) imply the bounds
$$
\| \partial_x Z_{\pm}^{(m+1)} \| \leq (m+1) C_{\varphi} + C_h T \leq (m+2) C_{\varphi}, \quad
\| \partial_x Y_{\pm}^{(m+1)} \| \leq \| \partial_x Z_{\mp}^{(m+1)} \| \leq (m+2) C_{\varphi},
$$
which yields the last two bounds in (\ref{bound-on-approximations-after}).
\end{proof}

With Lemma \ref{lemma-continuation}, we finally extend the local solution to every $T > 0$ and thus
prove Theorem \ref{theorem-main-global}. By Lemma \ref{lemma-continuation} and the induction method, we construct a sequence
of local solutions $\{ (Z_{\pm}^{(m)},Y_{\pm}^{(m)}) \}_{m \in \mathbb{N}} \in C^{1,1,1}(\Gamma_{T_m-T_{m-1}})$
to the system of integral equations (\ref{2.11})--(\ref{2.14}). The sequence is extended to
the time $T_m$, which is obtained from (\ref{constraint-time-after}) as
\begin{equation}
T_m = \sum_{k=1}^m T_k - T_{k-1} =  \sum_{k=1}^m \frac{1}{15 k C_\varphi},
\end{equation}
where we assumed $C_h \leq 15 C_{\varphi}^2$ for simplicity.
Since the harmonic series $\sum_{k=1}^{\infty} \frac{1}{k}$ diverges, the sequence of local
solutions is extended to arbitrary time $T > 0$ by incrementing the values of $m$.

By Proposition \ref{prop-correspondence}, we obtain the classical solution to system (\ref{sym-system})
by $z_{\pm}(t,x) = Z_{\pm}(t;t,x)$ for every $(t,x) \in \Gamma_T$ and every $T > 0$.
Using the transformation formulas (\ref{RH-inverse}), we obtain the classical solution
$(u,\eta)$ to the shallow water system (\ref{shallow-water}).
Thus, the proof of Theorem \ref{theorem-main-global} is complete.

\end{document}